\documentclass[12pt]{amsart}

\usepackage{amsmath,amssymb,amsfonts,amsthm}
\usepackage{ifpdf}

\usepackage{comment}

\usepackage[all]{xy}
\SelectTips{cm}{}
\CompileMatrices

\usepackage[section]{placeins}
\usepackage{leftidx}
\usepackage{stmaryrd} 
\usepackage{microtype}

\newcommand{\arxiv}[1]{\href{http://arxiv.org/abs/#1}{\tt arXiv:\nolinkurl{#1}}}

\newcommand{\googlebooks}[1]{(preview at \href{http://books.google.com/books?id=#1}{google books})}

\theoremstyle{plain}
\newtheorem{prop}{Proposition}[subsection]
\makeatletter
\@addtoreset{prop}{section}
\makeatother

\newtheorem{thm}[prop]{Theorem}

\newtheorem{lem}[prop]{Lemma}
\newtheorem*{lem*}{Lemma}
\newtheorem{cor}[prop]{Corollary}
\newtheorem*{cor*}{Corollary}

\newtheorem*{defn*}{Definition}             

\newtheorem*{question*}{Question}

\newenvironment{rem}{\vspace{0.3cm}\noindent\textsl{Remark.}}{}  
 
\newtheorem{rem*}[prop]{Remark}
\numberwithin{equation}{section}

\newcommand{\noop}[1]{}

\def\clap#1{\hbox to 0pt{\hss#1\hss}}

\renewcommand{\imath}{\mathfrak{i}}
\renewcommand{\jmath}{\mathfrak{j}}

\newcommand{\floor}[1]{\left\lfloor#1\right\rfloor}

\DeclareMathOperator{\End}{End}
\DeclareMathOperator{\Hom}{Hom}

\DeclareMathOperator{\Rad}{Rad}
\DeclareMathOperator{\ad}{ad}
\DeclareMathOperator{\Ker}{Ker}
\DeclareMathOperator{\add}{add}
\DeclareMathOperator{\Ind}{Ind}

\newcommand{\C}{\mathbb{C}}
\newcommand{\Z}{\mathbb{Z}}
\newcommand{\N}{\mathbb{N}}

\ifpdf
	\usepackage[pdftex,plainpages=false,hypertexnames=false,pdfpagelabels]{hyperref}
	\usepackage[pdftex]{graphicx}
\else
	\usepackage[plainpages=false,hypertexnames=false,pdfpagelabels]{hyperref}
	\usepackage{graphicx}
\fi

\usepackage{tikz}
\usetikzlibrary{shapes}
\usetikzlibrary{positioning}
\usetikzlibrary{backgrounds}
\usetikzlibrary{decorations,decorations.pathreplacing,decorations.markings}
\usetikzlibrary{fit,calc,through}
\usetikzlibrary{external}
\usetikzlibrary{knots}

\tikzstyle{mid>}=[decoration={markings, mark=at position 0.5 with {\arrow{>}}}, postaction={decorate}]
\tikzstyle{mid<}=[decoration={markings, mark=at position 0.5 with {\arrow{<}}}, postaction={decorate}]
\tikzstyle{upper>}=[decoration={markings, mark=at position 0.8 with {\arrow{>}}}, postaction={decorate}]
\tikzstyle{upper<}=[decoration={markings, mark=at position 0.8 with {\arrow{<}}}, postaction={decorate}]
\tikzstyle{lower>}=[decoration={markings, mark=at position 0.2 with {\arrow{>}}}, postaction={decorate}]
\tikzstyle{lower<}=[decoration={markings, mark=at position 0.2 with {\arrow{<}}}, postaction={decorate}]
\tikzset{amp/.style = {regular polygon, regular polygon sides=3,
              draw, fill=white, text width=1em,
              inner sep=1mm, outer sep=0mm,
              shape border rotate=0, scale=0.4}}
\tikzset{coamp/.style = {regular polygon, regular polygon sides=3,
              draw, fill=white, text width=1em,
              inner sep=1mm, outer sep=0mm,
              shape border rotate=180, scale=0.4}}

\newcommand{\Alt}[1]{{\textstyle\bigwedge^{#1}}}

\newcommand{\one}{1}

\def\sl{{\mathfrak{sl}}}
\def\sp{{\mathfrak{sp}}}
\def\Sp{{\mathcal{S}p}}

\def\g{{\mathfrak{g}}}

\def\gl{{\mathfrak{gl}}}
\def\dU{\dot{{\mathcal{U}}}}

\def\Rep{\mathcal{R}ep}
\def\Perm{\mathcal{P}erm}
\def\Prm{\mathcal{P}rm}

\usepackage{environ}
\usepackage{xargs}

\newcommandx{\NewEnvironx}[5][2,3]{%
  \expandafter\newcommandx\csname start#1\endcsname[#2][#3]{#4}%
  \NewEnviron{#1}{\csname start#1\expandafter\endcsname\BODY #5}}

\newcommand{\ladderX}{1.5}
\newcommand{\ladderY}{1.5}
\newcommand{\ladderR}{0.6}
\newcommand{\laddercoordinates}[2]{
\foreach \x in {0,...,#1} {
	\foreach \y in {0,...,#2} {
		\coordinate (l\x\y) at (\x * \ladderX, \y * \ladderY);
		\coordinate (u\x\y) at ($(l\x\y)+\ladderR*(0,\ladderY)$);
		\coordinate (d\x\y) at ($(l\x\y)+(0,\ladderY)-\ladderR*(0,\ladderY)$);
	}
}
}
\newcommand{\ladderEn}[5]{
\draw[mid>] (l#1#2) -- (d#1#2);
\draw[mid>] (d#1#2) -- ($(l#1#2)+(0,\ladderY)$) node[left] {#3};
\draw[mid>] ($(l#1#2)+(\ladderX,0)$) -- ($(u#1#2)+(\ladderX,0)$);
\draw[mid>] ($(u#1#2)+(\ladderX,0)$) -- ($(l#1#2)+(\ladderX,\ladderY)$) node[right] {#4};
\draw[mid>] (d#1#2) --node[above]{#5} ($(u#1#2)+(\ladderX,0)$);
}
\newcommand{\ladderE}[4]{\ladderEn{#1}{#2}{#3}{#4}{}}
\newcommand{\ladderFn}[5]{
\draw[mid>] (l#1#2) -- (u#1#2);
\draw[mid>] (u#1#2) -- ($(l#1#2)+(0,\ladderY)$) node[left] {#3};
\draw[mid>] ($(l#1#2)+(\ladderX,0)$) -- ($(d#1#2)+(\ladderX,0)$);
\draw[mid>] ($(d#1#2)+(\ladderX,0)$) -- ($(l#1#2)+(\ladderX,\ladderY)$) node[right] {#4};
\draw[mid>] ($(d#1#2)+(\ladderX,0)$) --node[above]{#5} (u#1#2);
}
\newcommand{\ladderF}[4]{\ladderFn{#1}{#2}{#3}{#4}{}}
\newcommand{\ladderIn}[3]{\draw[mid>] (l#1#2) -- +($#3*(0,\ladderY)$);}
\newcommand{\ladderI}[2]{\ladderIn{#1}{#2}{1}}

\NewEnvironx{ladder}[2]{%
  \begin{tikzpicture}[baseline=13*\ladderY*#2]\laddercoordinates{#1}{#2}}
{\end{tikzpicture}}

\newcommand{\fuse}[3]{\tikz[baseline=0.5cm]{
\coordinate (z1) at (0,0);
\coordinate (z2) at (1,0);
\coordinate (c) at (0.5,0.5);
\coordinate (e) at (0.5,1);
\draw[mid>] (z1) node[below] {$#1$} -- (c);
\draw[mid>] (z2) node[below] {$#2$} -- (c);
\draw[mid>] (c) -- (e) node[above] {$#3$};
}}
\newcommand{\fork}[3]{\tikz[baseline=0.5cm]{
\coordinate (z1) at (0,1);
\coordinate (z2) at (1,1);
\coordinate (c) at (0.5,0.5);
\coordinate (e) at (0.5,0);
\draw[mid<] (z1) node[above] {$#1$} -- (c);
\draw[mid<] (z2) node[above] {$#2$} -- (c);
\draw[mid<] (c) -- (e) node[below] {$#3$};
}}

\def\semicolon{;}
\def\applytolist#1{
    \expandafter\def\csname multi#1\endcsname##1{
        \def\multiack{##1}\ifx\multiack\semicolon
            \def\next{\relax}
        \else
            \csname #1\endcsname{##1}
            \def\next{\csname multi#1\endcsname}
        \fi
        \next}
    \csname multi#1\endcsname}

\def\calc#1{\expandafter\def\csname c#1\endcsname{{\mathcal #1}}}
\applytolist{calc}QWERTYUIOPLKJHGFDSAZXCVBNM;

\usepackage{color}

\usepackage{xcolor}
\definecolor{dark-red}{rgb}{0.7,0.25,0.25}
\definecolor{dark-blue}{rgb}{0.15,0.15,0.55}
\definecolor{medium-blue}{rgb}{0,0,0.65}
\hypersetup{
    colorlinks, linkcolor={dark-red},
    citecolor={dark-blue}, urlcolor={medium-blue}
}

\title[Presentations of categories of modules using the CKM principle]{Presentations of categories of modules using the Cautis-Kamnitzer-Morrison principle}

\date{} 
\usepackage{enumitem}

\author{Giulian Wiggins}
\address{School of Mathematics and Statistics\\ University of Sydney \\ Australia}
\email{g.wiggins@maths.usyd.edu.au}

\begin{document}

\maketitle

\begin{abstract} 
We use duality theorems to obtain presentations of some categories of modules. To derive these presentations we generalize a result of Cautis-Kamnitzer-Morrison \cite{mainpaper}:

Let $\g$ be a reductive Lie algebra, and $A$ an algebra, both over $\C$. Consider a $(\g , A)$-bimodule $P$ in which 
\begin{enumerate}[label=(\alph*)]
\item $P$ has a multiplicity free decomposition into irreducible  $(\g , A)$-bimodules.
\item $P$ is ``saturated" i.e. for any irreducible $\g$-module $V$, if every weight of $V$ is a weight of $P$, then $V$ is a submodule of $P$.
\end{enumerate}
We show that statements (a) and (b) are necessary and sufficient conditions for the existence of an isomorphism of categories between the full subcategory of $\Rep A$ whose objects are $\g$-weight spaces of $P$, and a quotient of the category version of Lusztig's idempotented form, $\dU \g$, formed by setting to zero all morphisms factoring through a collection of objects in $\dU \g$ depending on $P$. This is essentially a categorical version of the identification of generalized Schur algebras with quotients of Lusztig's idempotented forms given in \cite{doty}.

Applied to Schur-Weyl Duality we obtain a diagrammatic presentation of the full subcategory of $\Rep S_d$ whose objects are direct sums of permutation modules, as well as an explicit description of the $\otimes$-product of morphisms between permutation modules.
Applied to Brauer-Schur-Weyl Duality we obtain diagrammatic presentations of subcategories of $\Rep \mathcal{B}_{d}^{(- 2n)}$ and $\Rep \mathcal{B}_{r,s}^{(n)}$ whose Karoubi completion is the whole of 
$\Rep \mathcal{B}_{d}^{(- 2n)}$ and $\Rep \mathcal{B}_{r,s}^{(n)}$ respectively. 
\end{abstract}

\tableofcontents

\section{Introduction}

Let $\g$ denote a reductive Lie algebra over $\C$. Let $\g = \mathfrak{s} \oplus \mathfrak{a}$ be the decomposition of $\g$ into semisimple and abelian parts. Let $\Delta = \{\alpha_1 , \ldots , \alpha_n \} $ be a choice of positive simple roots for $\mathfrak{s}$, and $\Lambda$ the integral weight lattice for $\mathfrak{g}$.
Let $\{e_i , f_i , h_i \}_{i = 1 , \ldots , n}$ be the Chevalley generators of $\mathfrak{s}$ with respect to $\Delta$, and $(a_{ij})$ the Cartan matrix of $\mathfrak{s}$ i.e. $h_i = \alpha^{\vee}_i$ and $a_{ij} = \alpha_{i} (h_j)$.

Consider the idempotented form of the enveloping algebra $U\g$, which we regard as the $\C$-linear category, $\dot{\mathcal{U}} \g$, with
\begin{itemize}
\item
Objects: Integral weights $\lambda \in \Lambda$.
\item
Morphisms: Let $1_{\lambda}$ denote the identity on $\lambda$. Let $1_{\mu} \dot{\mathcal{U}} \g 1_{\lambda}$ denote the space of morphisms from $\lambda$ to $\mu$. The morphisms are generated by 
$E_{i} 1_{\lambda} \in 1_{\lambda + \alpha_i} \dot{\mathcal{U}} \g  1_{\lambda}$ and $F_{i} 1_{\lambda} \in 1_{\lambda - \alpha_i} \dot{\mathcal{U}} \g  1_{\lambda}$ for $ 1 \leq i \leq n$. These morphisms satisfy the relations 
\begin{align}
E_i F_i 1_{\lambda} &= F_i E_i 1_{\lambda} + \lambda (h_i) 1_{\lambda}& 
\label{Ug1} \displaybreak[1]\\
E_i F_j 1_{\lambda} &= F_j E_i 1_{\lambda} &\text{if } i \neq j 
\label{Ug2} \displaybreak[1]\\
\ad(E_j)^{1-a_{ij}} E_i 1_{\lambda} &= 0&
\text{ and likewise with $F$ replacing $E$} 
\label{Ug3} 
\end{align}
\end{itemize} 

This paper only considers $\g$-modules with an integral weight space decomposition. For this reason we \emph{henceforth use the term $\g$-module to only refer to those $\g$-modules with an integral weight space decomposition}.

Let $Vec_{\C}$ denote the category of vector spaces over $\C$. Any $\g$-module $P$ is equivalent to a $\C$-linear functor $\mathcal{F}_P: \dU \g \rightarrow Vec_{\C}$ mapping objects $\lambda \in \Lambda$ to the $\lambda$-weight space of $P$, $P_{\lambda}$, and each morphism $E_i 1_{\lambda} \in 1_{\lambda + \alpha_i} \dot{\mathcal{U}} \g  1_{\lambda}$ to the map
\[
P_{\lambda} \rightarrow P_{\lambda + \alpha_i} : v \mapsto e_i v
\]
and likewise with the $F_i 1_{\lambda}$.

For any $\g$-module $P$ write $\Pi(P)$ for the set of weights of $P$.
Define  $\dU^{P} \g$ to be the $\C$-linear category defined
\begin{itemize}
\item
Objects: Weights of $P$.
\item
Morphisms: The space, $1_{\mu} \dot{\mathcal{U}}^P \g 1_{\lambda}$, of morphisms from $\lambda$ to $\mu$ is defined as the quotient of $1_{\mu} \dot{\mathcal{U}} \g 1_{\lambda}$ by the space of morphisms $f \in 1_{\mu} \dot{\mathcal{U}} \g 1_{\lambda}$ that factor through an object $\nu \notin \Pi(P)$.
\end{itemize}

Doty-Giaquinto-Sullivan \cite{doty4} introduced the term \textbf{saturated $\g$-module} to describe those $\g$-modules $P$ with the property that for every $\g$-dominant weight 
$\lambda \in \Pi (P)$, the irreducible module of highest weight $\lambda$ is a submodule of $P$. 

This paper gives two main results regarding saturated $\g$-modules:
\begin{itemize}
\item
(Doty Criterion, Lemma \ref{Sat}) The functor $\dU^{P} \g \rightarrow Vec_{\C}$ induced by $\mathcal{F}_P: \dU \g \rightarrow Vec_{\C}$ is faithful if and only if $P$ is saturated.
\item
(CKM Principle, Theorem \ref{CKM}) Let $A$ be an algebra over $\C$, and $P$ a finite dimensional $(\g, A)$-bimodule with multiplicity free decomposition
\begin{equation}\label{mf}
P = \bigoplus_{i \in \chi} V^{i} \otimes M^{i}
\end{equation}
where the $V^{i}$ are irreducible $\g$-modules, and the $M^{i}$ are irreducible $A$-modules.

If $P$ is a saturated $\g$-module, then the category $\dU^{P} \g$ is isomorphic to the full subcategory of $\Rep A$ whose objects are the $\g$-weight spaces of $P$. Moreover, if $A$ is semisimple and acts on $P$ faithfully then
$\Rep A$ is equivalent to the Karoubi completion of the additive closure of $\dU^{P} \g$ (Corollary \ref{karckm}). 
\end{itemize}

\begin{rem}
The Doty Criterion should be viewed as a categorical version of the correspondence between  generalized Schur algebras and certain quotients of Lusztig's idempotented forms given in
(\cite{doty}, Theorem 4.2). Indeed the Doty Criterion gives a condition for an isomorphism of algebras 
\begin{equation}\label{dotycorrespondence}
\bigoplus_{\lambda , \mu \in \Pi (P)} 1_{\mu} \dot{\mathcal{U}}^P \g  1_{\lambda}
\cong
\bigoplus_{\lambda , \mu \in \Pi (P)} \mathcal{F}_{P} (1_{\mu} \dot{\mathcal{U}}^P \g  1_{\lambda}).
\end{equation}
Whenever $P$ is saturated, the right hand side of (\ref{dotycorrespondence}) is the generalized Schur algebra $\mathbf{S}(\Pi (P))$ as defined in \cite{donkin}. Moreover this isomorphism is exactly that given by (\cite{doty}, Theorem 4.2). 
\end{rem}

The CKM Principle is the special case of the Doty Criterion in the case where $P$ is a $(\g, A)$-bimodule. Cautis-Kamnitzer-Morrison \cite{mainpaper} proved this result for (the quantum analog of) the $U(\gl_m) \otimes U(\sl_n)$
-module $\Alt{\bullet} (\C^n \otimes \C^m)$. This has a multiplicity free decomposition by skew-Howe duality. From this they give a presentation of the full subcategory of $\Rep \sl_n$
$\otimes$-generated by the fundamental representations of $\sl_n$.

In Section \ref{ckm} we give elementary proofs of the Doty Criterion and CKM Principle.

The following multiplicity-free commuting actions are saturated (\cite{doty4}, Section 6),
\begin{align}
\gl_n \curvearrowright &\textstyle\bigotimes^{d} \C^n \curvearrowleft \C[S_d]
\label{eg1}
\displaybreak[1]\\
\sp_{2n} \curvearrowright &\textstyle\bigotimes^{d} \C^{2n} \curvearrowleft \mathcal{B}_{d}^{(-2n)} 
\label{eg2}
\displaybreak[1]\\
\mathfrak{o}_{2n} \curvearrowright &\textstyle\bigotimes^{d} \C^{2n} \curvearrowleft \mathcal{B}_{d} ^{(2n)} 
\label{eg3}
\displaybreak[1]\\
\gl_{n}  \curvearrowright \textstyle\bigotimes^{r} \C^n &\otimes \textstyle\bigotimes^{s} \C^{*n} \curvearrowleft \mathcal{B}_{r,s}^{(n)}
\label{eg4}
\end{align}
where 
$\mathcal{B}_{d}^{\delta}$ is the Brauer algebra over $\C$ spanned by $d$-diagrams parametrized by $\delta$ (as defined in \cite{origin}), and
$\mathcal{B}_{r,s}^{(n)} \subset \mathcal{B}_{r+s}^{(n)}$ is the walled Brauer algebra (as defined in \cite{mixsw}). These definitions and the commuting actions in (\ref{eg2}), (\ref{eg3}), (\ref{eg4}) are recalled in Section \ref{BCKM}.

Say that a category, $\mathcal{D}$, is a \textbf{pre-Karoubi subcategory} of a category, $\mathcal{C}$, if $\mathcal{C}$ is equivalent to the additive closure of the Karoubi completion of $\mathcal{D}$. In this case we call any presentation of $\mathcal{D}$ a 
\textbf{pre-Karoubi presentation} of $\mathcal{C}$.

If $n \geq d$ then the right action in (\ref{eg1}) is faithful. We use the CKM principle to give a pre-Karoubi presentation of $\Rep S_d$ (Theorem \ref{equiv}). 

If $2n \geq d-1$ then the right actions in (\ref{eg2}) and (\ref{eg3}) are faithful \cite{brown}. We give a pre-Karoubi presentation of $\Rep  \mathcal{B}_{d}^{(-2n)}$ for $2n \geq d-1$ (Proposition \ref{symsw}). 

If $n \geq r+s$ then the right action in (\ref{eg4}) is faithful \cite{mixsw}. We give a  pre-Karoubi presentation of $\Rep \mathcal{B}_{r,s}^{(n)}$ for $n \geq r+s$ (Proposition \ref{msw}). 

\subsection{A diagrammatic presentation of $\Perm (S_d)$}
Let us explain in more detail the pre-Karoubi presentation of $\Rep S_d$.
Given a sequence  $\lambda = (\lambda_1, \ldots , \lambda_n) \in \Z^n$ and $d \in \Z$, write $\lambda \vDash d$ if $\lambda$ is a composition of $d$. For $\lambda \vDash d$ denote the Young subgroup
 \[
S_{\lambda} := S_{\{1 , \ldots , \lambda_1\}} \times S_{\{\lambda_1 +1 , \ldots , \lambda_1 +\lambda_2\}}
\times \cdots \times S_{\{\lambda_1 + \cdots + \lambda_{n-1} + 1 , \ldots , d \}}
\]

We give a diagrammatic presentation of the full subcategory, $\Perm (S_d)$, of $\Rep S_d$ whose objects are the right $\C[S_d]$ modules $M^{\lambda} = \C [S_{\lambda} \backslash S_d]$, where $\lambda \vDash d$ (Theorem \ref{equiv}). 
Indeed every $M^{\lambda}$, $\lambda \vDash d$, is isomorphic in $\Rep S_d$ to a $\gl_n$-weight space of $\bigotimes^{d} \C^d$. By applying the CKM-principle and the diagrammatic calculus developed in \cite{mainpaper}, we derive a presentation of 
$\Perm (S_d)$. We outline these results here.

First define the category $\Perm$ with
\begin{itemize}
\item
Objects: The permutation modules $M^{\lambda}_d := \C [S_{\lambda} \backslash S_d]$ for all $d \in \N$ and $\lambda \vDash d$. \item
Morphisms:
\[
\Hom (M^{\lambda}_{d} , M^{\mu}_{d'}) = 
\begin{cases}
\Hom_{S_d} (M^{\lambda}_{d} , M^{\mu}_{d}) & \text{if } d=d' \\
0  & \text{otherwise}
\end{cases}
\]
\end{itemize}

\begin{rem}
$\Perm$ is the full subcategory of $\bigoplus_{d=0}^{\infty} \Rep S_d$ whose objects are the $M^{\lambda}_{d}$. We omit the subscript in $M^{\lambda}_{d}$ when the value of $d$ is clear from the context.
\end{rem}

It is more natural from the diagrams perspective to first give a diagrammatic presentation of $\Perm$,  and then consider $\Perm (S_d)$ as a full subcategory of $\Perm$. For this we define a monoidal product on $\Perm$.

The category $\bigoplus_{d=0}^{\infty} \Rep S_d$ has a monoidal product defined on direct summands by the bifunctor
\[
\cdot \circ \cdot = \Ind_{S_d \times S_{d'}}^{S_{d+d'}}  (\cdot \boxtimes \cdot) : \Rep S_d \times \Rep S_{d'} \rightarrow \Rep S_{d+d'}
\]
This restricts to a monoidal product on $\Perm$ since
\begin{align*}
M^{\lambda}_{d} \circ M^{\mu}_{d'} 
&=\left( 
 \C [S_{\lambda} \backslash S_d] \otimes  \C [S_{\mu} \backslash S_{d'}]
\right)
\otimes_{\C [S_d \times S_{d'}]} \C [S_{d+d'}]
= \C[S_{\lambda} \times S_{\mu} \backslash S_{d+d'}] = M^{(\lambda, \mu)}_{d+d'}
\end{align*}
where $(\lambda, \mu)$ is the concatenation of the sequences $\lambda$ and $\mu$.
We regard $\Perm$ as a monoidal category with the monoidal product $\cdot \circ \cdot $. It will be shown (Theorem \ref{equiv}) that the morphisms of $\Perm$
are generated (as a monoidal $\C$-linear category) by the morphisms:

For $k,l$ nonnegative integers, and $G=S_{k,l} \backslash S_{k+l}$:
\begin{align*}
\nabla_{k,l}:& M^{k,l} \rightarrow \C = M^{k+l} ~;~ \sum_{g \in G}
k_g g \rightarrow \sum_{g \in G} k_g &\text{for $k_g \in \C$} 
\displaybreak[1]\\
\Delta_{k,l}: & \C = M^{k+l} \rightarrow M^{k,l} ~;~ 1 \rightarrow \sum_{g \in G} g 
=\frac{1}{k!l!} \sum_{g \in S_{k+l}} S_{k,l} g
\end{align*}

We depict these morphisms diagrammatically 
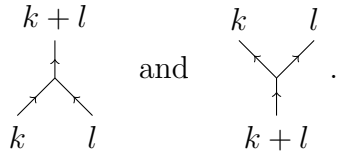
\begin{equation*}
\fuse{k}{l}{k+l} \ \ \text{ and } \ \ \fork{k}{l}{k+l}.
\end{equation*}
with composition drawn by vertical juxtaposition, and the $\circ$-product drawn by horizontal juxtaposition. Theorem \ref{equiv} states that $\Perm$ is the monoidal $\C$-linear category $\circ$-generated by the $\nabla_{k,l}$, $\Delta_{k,l}$ modulo the relations (\ref{eq2:IH}), (\ref{eq2:HI}), (\ref{eq2:bigon1}), (\ref{eq2:commute}), and any planar isotopy that keeps 
the edges oriented upwards. 

A presentation of $\Perm(S_d)$ can be obtained from this presentation by restricting to diagrams in which the numbers along the bottom (and top) sum to $d$. In particular, $\Perm(S_d)$ is generated by the morphisms, for 
$\lambda = (\lambda_1, \ldots , \lambda_n) \vDash d$,
\[
M_i 1_{\lambda} := 
\begin{tikzpicture}[baseline=20, xscale=0.75]
\laddercoordinates{8}{1}
\node[below] at (l00) {$\lambda_{1}$};
\node[below] at (l20) {$\lambda_{i{-}1}$};
\node[below] at (l50) {$\lambda_{i{+}2}$};
\node[below] at (l70) {$\lambda_{n}$};
\node at ($(l10)+(0,1)$) {$\cdots$};
\ladderI{0}{0};
\ladderI{2}{0};
\ladderI{5}{0};
\ladderI{7}{0};
\node at ($(l60)+(0,1)$) {$\cdots$};
\draw[mid>] ($(l30)$) node[below] {$\lambda_{i}$} -- ($(l30)+(0.75,0.75)$);
\draw[mid>] ($(l40)$) node[below] {$\lambda_{i{+}1}$} -- ($(l30)+(0.75,0.75)$);
\draw[mid>] ($(l30)+(0.75,0.75)$) -- ($(l30)+(0.75,1.5)$) node[above] {$\lambda_{i}+\lambda_{i{+}1}$};
\end{tikzpicture} 
\]
and
\[
M'_i 1_{\lambda} := 
\begin{tikzpicture}[baseline=20, xscale=0.75]
\laddercoordinates{8}{1}
\node[below] at (l00) {$\lambda_{1}$};
\node[below] at (l20) {$\lambda_{i{-}1}$};
\node[below] at (l50) {$\lambda_{i{+}2}$};
\node[below] at (l70) {$\lambda_{n}$};
\node at ($(l10)+(0,1)$) {$\cdots$};
\ladderI{0}{0};
\ladderI{2}{0};
\ladderI{5}{0};
\ladderI{7}{0};
\node at ($(l60)+(0,1)$) {$\cdots$};
\draw[mid<] ($(l31)$) node[above] {$\lambda_{i}$} -- ($(l30)+(0.75,0.75)$);
\draw[mid<] ($(l41)$) node[above] {$\lambda_{i{+}1}$} -- ($(l30)+(0.75,0.75)$);
\draw[mid>] ($(l30)+(0.75,0)$) node[below] {$\lambda_{i}+\lambda_{i{+}1}$}-- ($(l30)+(0.75,0.75)$) ;
\end{tikzpicture} 
\]
These are defined explicitly, for $\lambda' = (\lambda_1, \ldots, \lambda_{i} + \lambda_{i+1}, \ldots, \lambda_n)$,
\begin{align*}
M_i 1_{\lambda} : M^{\lambda} \rightarrow M^{\lambda'} &~;~ S_{\lambda} g \mapsto S_{\lambda'} g
\displaybreak[1]
\\
M'_{i} 1_{\lambda} : M^{\lambda'} \rightarrow M^{\lambda} &~;~ S_{\lambda'} g \mapsto \frac{1}{\lambda_i ! \lambda_{i+1} !} 
\sum_{h \in S_{\{ \lambda_{i-1} +1, \ldots, \lambda_{i+1} \}}} S_{\lambda} h g 
\end{align*}

The category $\Perm$ is symmetric monoidal, with the braiding isomorphism $M^{k,l} \rightarrow M^{l,k}$ defined diagrammatically (Theorem \ref{Prmweb}):
\[
\tikz[baseline=40]{
\draw[mid>] (0,0) 
to [out=up, in=down] (2,3);
\draw[mid>] (2,0) 
to [out=up, in=down] (0,3);
\node[left] at (0,0) {$k$};
\node[right] at (2,0) {$l$};
\node[left] at (0,3) {$l$};
\node[right] at (2,3) {$k$};
}
:=
\sum_{a,b \geq 0 , a-b=k-l} (-1)^{k-a} 
\begin{tikzpicture}[baseline=40]
\laddercoordinates{1}{2}
\node[left] at (l00) {$k$};
\node[right] at (l10) {$l$};
\ladderEn{0}{0}{$k{-}a$}{$l{+}a$}{a}
\ladderFn{0}{1}{$l$}{$k$}{b}
\end{tikzpicture}
\]

\subsection{The Kronecker product of permutation modules} It is is known that the tensor product of two permutation modules of $S_d$ decomposes as a direct sum of permutation modules of $S_d$ (\cite{jks}, lemma 2.9.16). In Section \ref{kronecker} we recall this decomposition, and give an explicit description of the $\otimes$-product of morphisms between permutation modules, in terms of our generating morphisms.

\subsection{Acknowledgements} 
Part of this research was undertaken for an honours thesis at the University of Sydney, under the supervision of Dr. Oded Yacobi.

I thank my mentor Dr. Oded Yacobi for suggesting this problem to me, and for the continual guidance and feedback in the research and writing of this paper.

I have also benefitted from Cautis, Kamnitzer, Morrison uploading the source file for \cite{mainpaper} onto the arXiv. Indeed many of the diagrams here use their macros.

\section{Cautis-Kamnitzer-Morrison Principle}\label{ckm}

In this section we prove the CKM principle.
Recall we use the term $\g$-module to just refer to $\g$-modules with an integral weight space decomposition. Let $\Lambda$ denote the integral weights of $\g$.
For any $\g$-module $V$, write $\Pi(V)$ for the set of weights of $V$, and $\Pi^{+} (V)$ for the set of $\g$-dominant weights of $V$. 
For any $\g$-dominant weight $\lambda \in \Lambda$, write $V^{\lambda}$ for the irreducible $\g$-module of highest weight $\lambda$. 

Say that a $\g$-module $P$ is \textbf{saturated (for $\g$)} if 
$\{ V^{\lambda} \}_{\lambda \in \Pi^{+} (P)}$ is the set of irreducible $\g$-submodules of $P$ up to isomorphism.

\begin{lem}\label{21.3}
A finite dimensional $\g$-module $P$ is saturated if and only if for any irreducible $\g$-module $V$, $\Pi(V) \subseteq \Pi(P)$ implies $V$ is a submodule of $P$. 
\end{lem}

\begin{proof}
Define the usual partial order on $\Lambda$: $\lambda \geq \mu$ if $\lambda - \mu$ is a sum of positive roots. 
Let $\mathcal{W}$ be the Weyl group of $\g$. If $\lambda \in \Lambda$ is $\g$-dominant then
\[
\Pi (V^{\lambda}) = \{
\mu \in \Lambda | \text{if $\nu \in \mathcal{W} \mu$ then $\nu \leq \lambda$}
\}
\]
Hence $\lambda \in \Pi^{+}(P)$ if and only if $\Pi(V^{\lambda}) \subseteq \Pi(P)$. The result follows.
\end{proof}

Define the algebra version of Lusztig's idempotented form,
\[
\dot{U} \g := \bigoplus_{\lambda , \mu \in \Lambda} 1_{\mu} \dot{\mathcal{U}} \g  1_{\lambda}
\]
with multiplication defined by composition in $\dU \g$. Define the two-sided ideal, $I_{P}$, of $\dot{U} \g$,
\[
I_P := \langle 1_{\lambda} | \lambda \notin \Pi(P) \rangle
\]
Recall that a $\g$-module $P$ is equivalent to a functor $\mathcal{F}_P: \dU \g \rightarrow Vec_{\C}$, and this functor factors through the category $\dU^{P} \g$ defined by setting to zero all morphisms that factor through an object $\lambda \notin \Pi(P)$. Such a functor $\mathcal{F}_P: \dU \g \rightarrow Vec_{\C}$ is equivalent to a representation map $\dot{F}_P: \dot{U} \g \rightarrow \End P$. Clearly $\dot{F}_P$ factors through the algebra
\begin{align*}
\dot{U}^P \g &:= \bigoplus_{\lambda , \mu \in \Pi(P)} 1_{\mu} \dot{\mathcal{U}}^P \g 1_{\lambda} \displaybreak[1]\\
&= \dot{U} \g / I_P
\end{align*}

It is shown in (\cite{doty}, Theorem 4.2) that the algebra $\dot{U}^P \g$ is isomorphic to the generalized Schur algebra $\mathbf{S}(\Pi (P)) := \dot{U} \g / \Ker \dot{F}_P$. The following lemma is shown in 
\cite{doty} and we include our own proof for completeness.

\begin{lem}\label{Wed}
If $P$ is a finite dimensional saturated $\g$-module then 
\begin{enumerate}
\item
$\{ V^{\lambda} \}_{\lambda \in \Pi^{+} (P)}$ is the set of irreducible $\dot{U}^P \g $-modules up to isomorphism.
\item
$\dot{U}^P \g $ is finite dimensional and semisimple.
\end{enumerate}
\end{lem}

\begin{proof}
The (irreducible) $ \dot{U}^P \g$ modules are the (irreducible)
$\g$-modules $V$ such that $\Pi(V) \subseteq \Pi(P)$. By Lemma \ref{21.3} the irreducible $ \dot{U}^P \g$ modules are precisely the irreducible submodules $\{ V^{\lambda} \}_{\lambda \in \Pi^{+} (P)}$. By the PBW theorem, and since $P$ has finitely many weights, $ \dot{U}^P \g$ is finite dimensional. Since  $ \g $-modules are completely reducible, $ \dot{U}^P \g$ is semisimple. 
\end{proof}

The kernel of the algebra representation map $\dot{F}_P : \dot{U} \g \rightarrow \End P$ is the two-sided ideal in $\dot{U} \g$,
\[
\Ker \dot{F}_P := \langle f \in 1_{\mu} \dot{\mathcal{U}} \g  1_{\lambda} |  
\mathcal{F}_P (f) = 0 \in \Hom_{\C} (P_{\lambda}, P_{\mu}), \lambda, \mu \in \Lambda
 \rangle
\]

The following lemma is a direct consequence of (\cite{doty}, Theorem 4.2). We include a more elementary proof.
\begin{lem}[Doty Criterion]\label{Sat}
A finite dimensional $\g$-module $P$ is saturated if and only if
$
\Ker \dot{F}_P = I_P
$. 
\end{lem}

\begin{proof}
Suppose $P$ is saturated.
Clearly $I_P \subseteq \Ker \dot{F}_P$. To prove the converse it suffices to show that $ \dot{U} \g / I_P$ acts on $P$ faithfully. By Lemma \ref{Wed},
\begin{align*}
\{ f \in \dot{U}^P \g | fP = 0 \} &= \Rad(\dot{U}^P \g) 
&\text{by Lemma \ref{Wed} (1)} 
\displaybreak[1]\\
&= 0 &\text{by Lemma \ref{Wed} (2)} 
\end{align*}
as required.

If P is not saturated then Wedderburn's theorem gives a decomposition
\[
\dot{U}^P \g = \left( \bigoplus_{\lambda \in \Pi^{+} (P)} \End V^{\lambda} \right) \oplus 
\left( \bigoplus_{i \in \chi} \End W^i \right)
\]
where the $W^i$ are nonzero irreducible $\g$-modules not contained in $P$ (but whose weights are weights of $P$). Then each $\End W^i \subset \Ker \dot{F}_P$ but $\End W^i  \not\subset I_P$.
\end{proof}

\begin{rem}
As an example of Lemma \ref{Sat} take $\g = \sl_2 (\C)$ and $P= V(2)$, the irreducible $\sl_2$-module of highest weight 2. Then
$z := E F1_0 +F  E1_0 - 41_0 \in 1_0 \dU  \sl_2 1_0$ is in $\Ker \dot{F}_P$ and not in $I_P$. If we instead take $P$ to be the saturated module $V(2) \oplus V(0)$ then $z \notin \Ker \dot{F}_P$ as expected. 
\end{rem}

\begin{thm}[CKM Principle]\label{CKM}
Let $A$ be an algebra over $\C$, and $P$ a finite dimensional $(\g, A)$-bimodule that has a multiplicity free decomposition
\[
P = \bigoplus_{\lambda \in \Pi^{+}(P)} V^{\lambda} \otimes L^{\lambda}
\]
where $\{L^{\lambda} \}_{\lambda \in \Pi^{+} (P)}$ is some family of irreducible $A$-modules. Then $\dot{U}^P \g \cong \End_A P$ and $\dU^P \g$ is isomorphic to the full subcategory, $\Rep^{P} A$, of $\Rep A$ whose objects are the weight spaces of $P$.
\end{thm}

\begin{proof} By Lemma \ref{Sat}, the $\C$-linear functor $\mathcal{F}_P: \dU \g \rightarrow Vec_{\C}$ factors through a faithful functor $\dU^P \g \rightarrow \Rep A$. This functor is fully-faithful since we have the following isomorphism of algebras,
\begin{align*}
\bigoplus_{\lambda , \mu \in \Pi(P)} \Hom_A (P_{\lambda} , P_{\mu})
&\cong \End_A P &\displaybreak[1]\\
 &= \End_A (\bigoplus_{\lambda \in \Pi^{+} (P)} V^{\lambda} \otimes L^{\lambda}) &\displaybreak[1]\\
&\cong \bigoplus_{\lambda \in \Pi^{+} (P)} \End_{A} (V^{\lambda} \otimes L^{\lambda}) &\text{ by Schur's Lemma  and multiplicity freeness} \displaybreak[1]\\
&\cong \bigoplus_{\lambda \in \Pi^{+} (P)} \End V^{\lambda} \otimes \End_A L^{\lambda} & \displaybreak[1]\\
&\cong \bigoplus_{\lambda \in \Pi^{+} (P)} \End V^{\lambda} &\text{ by Schur's Lemma} \displaybreak[1]\\
&\cong \dot{U}^P \g  &\text{ by Lemma \ref{Wed} and Wedderburn's Theorem}
\end{align*}
The result follows.
\end{proof}

Given a linear category $\mathcal{C}$, define the \textbf{free additive category generated by $\mathcal{C}$}, $\add$-$\mathcal{C}$, to be the category whose objects are biproducts of objects in $\mathcal{C}$, as well as a zero object. Morphisms 
$f: A_1 \oplus \cdots \oplus A_n \rightarrow B_1 \oplus \cdots \oplus B_m$ in $\add$-$\mathcal{C}$ are written as $m \times n$-matrices
\[
f = \begin{pmatrix}
    f_{11} & f_{12} & f_{13} & \dots  & f_{1n} \\
    f_{21} & f_{22} & f_{23} & \dots  &f_{2n} \\
    \vdots & \vdots & \vdots & \ddots & \vdots \\
    f_{m1} & f_{m2} & f_{m3} & \dots  & f_{mn}
\end{pmatrix}
\]
where $f_{ij}: A_j \rightarrow B_i$ is a morphism in $\mathcal{C}$. Addition and composition of morphisms in
$\add$-$\mathcal{C}$ is given by usual matrix addition and multiplication. 

The \textbf{Karoubi envelope of $\mathcal{C}$} has as objects pairs $(A, e)$ where $A$ is an object in $\mathcal{C}$ and $e:A \rightarrow A$ is an idempotent morphism in $\mathcal{C}$. Morphisms 
$(A,e) \rightarrow (B,f)$ are morphisms $\varphi:A \rightarrow B$ in $\mathcal{C}$ such that $\varphi=f \varphi e$.

\begin{cor}\label{karckm}
Assume the same set up as Theorem \ref{CKM}. Suppose $A$ is finite dimensional and semisimple. If $A$ acts on $P$ faithfully then the Karoubi envelope of $\add$-$\dU^{P} \g$ is equivalent to $\Rep A$.
\end{cor}

\begin{proof}
By the CKM Principle $\add$-$\dU^{P} \g$ is isomorphic to $\add$-$\Rep^{P} A$, which is equivalent to the full subcategory of $\Rep A$ whose objects are direct sums of $\g$-weight spaces of $P$. Now, $P$ has the following decompositions into $A$-modules,
\[
P=\bigoplus_{\lambda \in \Pi(P)} P_{\lambda} = \bigoplus_{\lambda \in \Pi^{+}(P)} (\dim V^{\lambda} ) L^{\lambda} .
\]
Hence each of the $L^{\lambda}$ appear as a direct summand of one of the $P_{\lambda}$. Hence it is enough to show that $\{L^{\lambda} \}_{\lambda \in \Pi^{+} (P)}$ is the set of all irreducible $A$-modules up to isomorphism.
Note that we can identify,
\begin{align*}
A &\subseteq \End_{\g} (P) &\text{since $A$ acts faithfully on $P$} 
\displaybreak[1] \\
&=\bigoplus_{\lambda \in \Pi^{+} (P)} \End L^{\lambda} &\text{as in the proof of Theorem \ref{CKM}}
\displaybreak[1] \\
&\subseteq A &\text{by Wedderburn's Theorem}
\end{align*}
It follows that $\{L^{\lambda} \}_{\lambda \in \Pi^{+} (P)}$ is the set of all irreducible $A$-modules up to isomorphism.
The result follows.
\end{proof}

\section{A Diagrammatic Presentation of $\Perm$}\label{Perm}

As a first application of the CKM Principle, we use Schur-Weyl duality to derive a diagrammatic presentation of the full subcategory, $\Perm$, of $\bigoplus_{d=0}^{\infty} \Rep S_d$ whose objects are the permutation modules, $M^{\lambda}$, for $\lambda \vDash d$, $d \in \N$ (Theorem \ref{equiv}). 

\subsection{The category of $\Perm$-spiders}

Define a \textbf{$\Perm$-spider} to be a planar diagram made up of $\Z_{>0}$-labelled strands whose endpoints either intersect the bottom or top of the diagram, or intersect a vertex of the form
\[
\fuse{k}{l}{k+l} \ \ \text{ or } \ \ \fork{k}{l}{k+l}.
\]

We identify $\Perm$-spiders up to any planar isotopy that preserves the upwards orientation of the edges. Given two sequences $\lambda$, $\mu$ of integers we say that a $\Perm$-spider \textbf{connects $\lambda$ to $\mu$} if the sequence of edges along the bottom (respectively top) of the diagram is $\lambda$ (respectively $\mu$).

We draw $\Perm$-spiders with strands labelled by any integer. Strands with a 0 label are interpreted by removing the strand 
from the diagram. Strands with a non-positive label are zero morphisms.

Define the $\C$-linear monoidal category $\Sp (\Perm)$:
\begin{itemize}
\item
Objects: Finite sequences of nonnegative integers. The monoidal product on objects is given by concatenation.
\item
Morphisms: Morphisms from $\lambda$ to $\mu$ are $\C$-linear combinations of $\Perm$-spiders connecting $\lambda$ to $\mu$. Composition (respectively the monoidal product) is given by vertical (respectively horizontal) juxtaposition of diagrams. These satisfy the following relations 
\begin{align}
\begin{tikzpicture}[baseline]
\foreach \x/\y in {0/0,1/0,2/0,0/1,1/1,0/2} {
	\coordinate(z\x\y) at (\x+\y/2,\y/1.5);
}
\coordinate (z03) at (1,2);
\draw[mid>] (z00) node[below] {$k$} --  (z01);
\draw[mid>] (z01) -- node[left] {$k+l$} (z02);
\draw[mid>] (z10) node[below] {$l$} -- (z01);
\draw[mid>] (z20) node[below] {$m$} -- (z02);
\draw[mid>](z02) -- node[left] {$k+l+m$} (z03);
\end{tikzpicture}
& =
\begin{tikzpicture}[baseline]
\foreach \x/\y in {0/0,1/0,2/0,0/1,1/1,0/2} {
	\coordinate(z\x\y) at (\x+\y/2,\y/1.5);
}
\coordinate (z03) at (1,2);
\draw[mid>] (z00) node[below] {$k$} --  (z02);
\draw[mid>] (z10) node[below] {$l$} -- (z11);
\draw[mid>] (z20) node[below] {$m$} -- (z11);
\draw[mid>] (z11) -- node[right] {$l+m$} (z02);
\draw[mid>](z02) -- node[left] {$k+l+m$} (z03);
\end{tikzpicture}
\label{eq2:IH}
\displaybreak[1]
\\
\label{eq2:HI}
\begin{tikzpicture}[baseline]
\foreach \x/\y in {0/0,1/0,2/0,0/1,1/1,0/2} {
	\coordinate(z\x\y) at (-\x-\y/2+1,-\y/1.5+1);
}
\coordinate (z03) at (0,-1);
\draw[mid<] (z00) node[above] {$m$} --  (z01);
\draw[mid<] (z01) -- node[right] {$l+m$} (z02);
\draw[mid<] (z10) node[above] {$l$} -- (z01);
\draw[mid<] (z20) node[above] {$k$} -- (z02);
\draw[mid<](z02) -- node[left] {$k+l+m$} (z03);
\end{tikzpicture}
& =
\begin{tikzpicture}[baseline]
\foreach \x/\y in {0/0,1/0,2/0,0/1,1/1,0/2} {
	\coordinate(z\x\y) at (-\x-\y/2+1,-\y/1.5+1);
}
\coordinate (z03) at (0,-1);
\draw[mid<] (z00) node[above] {$m$} --  (z02);
\draw[mid<] (z10) node[above] {$l$} -- (z11);
\draw[mid<] (z20) node[above] {$k$} -- (z11);
\draw[mid<] (z11) -- node[left] {$k+l$} (z02);
\draw[mid<](z02) -- node[right] {$k+l+m$} (z03);
\end{tikzpicture}
\displaybreak[1] 
\\
\begin{tikzpicture}[baseline=20]
\foreach \n in {0,...,3} {
	\coordinate (z\n) at (0.4*\n, 0.8*\n);
}
\draw[mid>] (z0) -- node[right] {$k+l$} (z1);
\draw[mid>] (z2) -- node[right] {$k+l$} (z3);
\draw[mid>] (z1) to[out=150,in=-190] node[left] {$k$} (z2);
\draw[mid>] (z1) to[out=-30,in=0] node[right] {$l$} (z2);
\end{tikzpicture}
& = \binom{k+l}{l}
\tikz[baseline=20]{\draw[mid>] (0,0) -- node[right] {$k+l$} (1,2);}
\label{eq2:bigon1}
\displaybreak[1] 
\\
\begin{tikzpicture}[baseline=40]
\laddercoordinates{1}{2}
\node[left] at (l00) {$k$};
\node[right] at (l10) {$l$};
\ladderEn{0}{0}{$k{-}s$}{$l{+}s$}{$s$}
\ladderFn{0}{1}{$k{-}s{+}r$}{$l{+}s{-}r$}{$r$}
\end{tikzpicture}
&= \sum_t \binom{k-l+r-s}{t}
\begin{tikzpicture}[baseline=40]
\laddercoordinates{1}{2}
\node[left] at (l00) {$k$};
\node[right] at (l10) {$l$};
\ladderFn{0}{0}{$k{+}r{-}t$}{$l{-}r{+}t$}{$r{-}t$}
\ladderEn{0}{1}{$k{-}s{+}r$}{$l{+}s{-}r$}{$s{-}t$}
\end{tikzpicture}
\label{eq2:commute}
\end{align}
\end{itemize}
where $k,l,r,s \in \Z$.

As an example, let us show that the following equation holds in $\Sp(\Perm)$.
\begin{equation}
\label{band}
\tikz[baseline=40]{
\laddercoordinates{1}{2}
\ladderEn{0}{0}{$k-s$}{$l+s$}{$s$}
\ladderEn{0}{1}{$k-s-r$}{$l+s+r$}{$r$}
\node[left] at (l00) {$k$};
\node[right] at (l10) {$l$};
}
=
\binom{r+s}{r}
\tikz[baseline=20]{
\laddercoordinates{1}{1}
\ladderEn{0}{0}{$k-s-r$}{$l+s+r$}{$r+s$}
\node[left] at (l00) {$k$};
\node[right] at (l10) {$l$};
}
\end{equation}
Indeed equation (\ref{band}) follows by applying (\ref{eq2:HI}) to the left and right upright strands, then applying (\ref{eq2:bigon1}) to the resulting bigon.

Define the following morphisms in $\Perm$:
\begin{align*}
\nabla_{k,l}:& M^{k,l} \rightarrow \C = M^{k+l} ~;~ \sum_{g \in G}
k_g g \mapsto \sum_{g \in G} k_g &\text{for $k_g \in \C$} \\
\Delta_{k,l}: & \C = M^{k+l} \rightarrow M^{k,l} ~;~ 1 \mapsto \sum_{g \in S_{k,l} / S_{k{+}l}} g = \frac{1}{k!l!} \sum_{g \in S_{k{+}l}} S_{k,l} g
\end{align*}

\begin{thm}\label{equiv} There is an isomorphism of $\C$-linear monoidal categories  $\Gamma: \Sp (\Perm) \rightarrow \Perm$ defined
\[
\fuse{k}{l}{k{+}l} \mapsto \nabla_{k,l}: M^{k,l} \rightarrow M^{k{+}l} \ \ \text{ , } \ \ \fork{k}{l}{k{+}l} \mapsto \Delta_{k,l}: M^{k{+}l} \rightarrow M^{k,l}
\]
In particular, the morphisms of $\Perm$ are $\circ$-generated by the morphisms $\nabla_{k,l}$ and $\Delta_{k,l}$.
\end{thm}

The proof of Theorem \ref{equiv} will cover Sections \ref{Ugln},  \ref{Snd}, and \ref{proof}.
In Section \ref{Ugln} we give a diagrammatic presentation of $\dU \gl_n$. In Section \ref{Snd} we apply Schur-Weyl Duality and the CKM Principle to identify the category, $\dot{\mathcal{S}} (n,d)$, of weight spaces of $\bigotimes^d \C^n$ with a quotient of $\dU \gl_n$. In Section \ref{proof} we use this presentation of $\dot{\mathcal{S}} (n,d)$ to construct a presentation of $\Perm$.

 \subsection{The category $\dU \gl_n$}\label{Ugln}
Fix the Cartan subalgebra, $\mathfrak{h}$, of $\gl_n$ to be the space of diagonal matrices. We write integral weights of $\gl_n$ by their vector coordinates with respect to the basis of $\mathfrak{h}^*$ dual to the standard basis $e_{11}, \ldots , e_{nn}$ of $\mathfrak{h}$. Let $\alpha_i $ be the sequence $(0, \ldots , 0 , 1 , -1, 0 , \ldots , 0) \in \Z^n$,
where the $1$ is in the $i$-th position. 

The $\C$-linear category $\dot{\mathcal{U}} \mathfrak{gl}_n $ is defined:
\begin{itemize}
\item
Objects: Sequences $\lambda = (\lambda_{1} , \ldots , \lambda_{n}) \in \Z^n $.
\item
Morphisms: 
The morphisms are generated by $E^{(r)}_{i} 1_{\lambda} \in 1_{\lambda + r \alpha_i} \dot{\mathcal{U}} \mathfrak{gl}_n  1_{\lambda}$ and $F^{(r)}_{i} 1_{\lambda} \in 1_{\lambda - r \alpha_i} \dot{\mathcal{U}} \mathfrak{gl}_n  1_{\lambda}$ for $ 1 \leq i \leq n-1$ and satisfy the relations:
\begin{align}
\label{rel:1}
\hspace*{-3.5cm} E_i^{(r)} F_i^{(s)} 1_{\lambda} &= \sum_t \binom{\lambda_i -\lambda_{i+1} + r - s}{t} F_i^{(s-t)} E_i^{(r-t)}1_{\lambda}&
\displaybreak[1]\\
\label{rel:2}  E_i^{(r)} F_j^{(s)} 1_{\lambda} &= F_j^{(s)} E_i^{(r)} 1_{\lambda}, & \text{ if $i \ne j$} \displaybreak[1]\\
\label{rel:3} E_iE_jE_i 1_{\lambda} &= (E_i^{(2)} E_j + E_j E_i^{(2)}) 1_{\lambda}, & \text{ if $ |i - j| = 1 $} \displaybreak[1]\\
\label{rel:4} E_i^{(r)} E_j^{(s)} 1_{\lambda} &= E_j^{(s)} E_i^{(r)} 1_{\lambda} &\text{ if $ |i-j| > 1$} \displaybreak[1]\\
\label{rel:5} E_i^{(s)} E_i^{(r)} 1_{\lambda}&= \binom{r+s}{r} E_i^{(r+s)} 1_{\lambda} &  
\end{align}
and likewise with $F$ and $E$ interchanged in equations (\ref{rel:3}), (\ref{rel:4}), (\ref{rel:5}).
\end{itemize}

The generating morphisms $E_i 1_{\lambda}$ and $F_i 1_{\lambda}$ correspond to the Chevalley generators $E_i = e_{i,i+1}$ and 
$F_i = e_{i+1, i}$ for $ 1 \leq i \leq n-1$.
The generating morphisms $E^{(r)}_{i} 1_{\lambda}$ and $F^{(r)}_{i} 1_{\lambda}$ for $r >1$ are the divided powers
\begin{align*}
E^{(r)}_{i} 1_{\lambda} = \frac{1}{r!} E^{r}_i 1_{\lambda} \qquad \text{ and } \qquad F^{(r)}_{i} 1_{\lambda} = \frac{1}{r!} F^{r}_i 1_{\lambda}
\end{align*}
We give the presentation of  $\dot{\mathcal{U}} \mathfrak{gl}_n $ in terms of these redundant generators because they are more amenable to diagrammatic calculus.

Cautis-Kamnitzer-Morrison \cite{mainpaper} identified morphisms in $\dU \gl_n$ with ``ladder" diagrams. 
We define an \textbf{$n$-ladder} to be an oriented plane diagram consisting of $n$ upwards oriented vertical strands labelled by integers, and (horizontal) strands connecting adjancent uprights and labeled by nonnegative integers. Following \cite{mainpaper} we identify morphisms in $\dU \gl_n$ with diagrams in the following way,

\begin{align}\label{glnladder}
E_i^{(r)} \one_{\lambda} &=
\begin{tikzpicture}[baseline=20, scale=0.75]
\laddercoordinates{3}{1}
\node[below] at (l00) {$\lambda_{i{-}1}$};
\node[below] at (l10) {$\lambda_i$};
\node[below] at (l20) {$\lambda_{i{+}1}$};
\node[below] at (l30) {$\lambda_{i{+}2}$};
\node[above] at (l11) {$\lambda_i{+}r$};
\node[above] at (l21) {$\lambda_{i{+}1}{-}r$};
\node at ($(l00)+(-1,1)$) {$\cdots$};
\ladderI{0}{0};
\ladderFn{1}{0}{}{}{$r$}
\ladderI{3}{0};
\node at ($(l30)+(1,1)$) {$\cdots$};
\end{tikzpicture} 
\qquad
\text{and}
\qquad
F_i^{(r)} \one_{\lambda} &=
\begin{tikzpicture}[baseline=20, scale=0.75]
\laddercoordinates{3}{1}
\node[below] at (l00) {$\lambda_{i{-}1}$};
\node[below] at (l10) {$\lambda_i$};
\node[below] at (l20) {$\lambda_{i{+}1}$};
\node[below] at (l30) {$\lambda_{i{+}2}$};
\node[above] at (l11) {$\lambda_i{-}r$};
\node[above] at (l21) {$\lambda_{i{+}1}{+}r$};
\node at ($(l00)+(-1,1)$) {$\cdots$};
\ladderI{0}{0};
\ladderEn{1}{0}{}{}{$r$}
\ladderI{3}{0};
\node at ($(l30)+(1,1)$) {$\cdots$};
\end{tikzpicture} 
\end{align}
with composition of morphisms drawn by vertical juxtaposition. The relations (\ref{rel:1})-(\ref{rel:5}) simplify to invariance under any planar isotopy that preserves the upward-orientation of the strands, as well as the following relations, 
\begin{align}
\begin{tikzpicture}[baseline=40,yscale=0.75]
\laddercoordinates{2}{2}
\node[below] at (l00) {$k_1$};
\node[below] at (l10) {$k_2$};
\node[below] at (l20) {$k_3$};
\ladderEn{0}{0}{$k_1{-}r$}{$k_2{+}r$}{$r$}
\ladderFn{1}{1}{$k_2{+}r{+}s$}{$k_3{-}s$}{$s$}
\ladderIn{0}{1}{1}
\ladderIn{2}{0}{1}
\end{tikzpicture}
&=
\begin{tikzpicture}[baseline=40,yscale=0.75]
\laddercoordinates{2}{2}
\node[below] at (l00) {$k_1$};
\node[below] at (l10) {$k_2$};
\node[below] at (l20) {$k_3$};
\ladderEn{0}{1}{$k_1{-}r$}{$k_2{+}r{+}s$}{$r$}
\ladderFn{1}{0}{$k_2{+}s$}{$k_3{-}s$}{$s$}
\ladderIn{0}{0}{1}
\ladderIn{2}{1}{1}
\end{tikzpicture}
\label{eq:IHlad}
\displaybreak[1] \\
\begin{tikzpicture}[baseline=40,yscale=0.75]
\laddercoordinates{2}{2}
\node[below] at (l00) {$k_1$};
\node[below] at (l10) {$k_2$};
\node[below] at (l20) {$k_3$};
\ladderFn{0}{0}{$k_1{+}r$}{$k_2{-}r$}{$r$}
\ladderEn{1}{1}{$k_2{-}r{-}s$}{$k_3{+}s$}{$s$}
\ladderIn{0}{1}{1}
\ladderIn{2}{0}{1}
\end{tikzpicture}
&=
\begin{tikzpicture}[baseline=40,yscale=0.75]
\laddercoordinates{2}{2}
\node[below] at (l00) {$k_1$};
\node[below] at (l10) {$k_2$};
\node[below] at (l20) {$k_3$};
\ladderFn{0}{1}{$k_1{+}r$}{$k_2{-}r{-}s$}{$r$}
\ladderEn{1}{0}{$k_2{-}s$}{$k_3{+}s$}{$s$}
\ladderIn{0}{0}{1}
\ladderIn{2}{1}{1}
\end{tikzpicture}
\label{eq:IHlad2}
\displaybreak[1] \\
\tikz[baseline=40,yscale=0.75]{
\laddercoordinates{1}{2}
\ladderEn{0}{0}{$k-s$}{$l+s$}{$s$}
\ladderEn{0}{1}{$k-s-r$}{$l+s+r$}{$r$}
\node[left] at (l00) {$k$};
\node[right] at (l10) {$l$};
}
&=
\binom{r+s}{r}
\tikz[baseline=20]{
\laddercoordinates{1}{1}
\ladderEn{0}{0}{$k-s-r$}{$l+s+r$}{$r+s$}
\node[left] at (l00) {$k$};
\node[right] at (l10) {$l$};
}
\label{eq:EE}
\displaybreak[1]
\\
\begin{tikzpicture}[baseline=40]
\laddercoordinates{1}{2}
\node[left] at (l00) {$k$};
\node[right] at (l10) {$l$};
\ladderEn{0}{0}{$k{-}s$}{$l{+}s$}{$s$}
\ladderFn{0}{1}{$k{-}s{+}r$}{$l{+}s{-}r$}{$r$}
\end{tikzpicture}
&= \sum_t \binom{k-l+r-s}{t}
\begin{tikzpicture}[baseline=40]
\laddercoordinates{1}{2}
\node[left] at (l00) {$k$};
\node[right] at (l10) {$l$};
\ladderFn{0}{0}{$k{+}r{-}t$}{$l{-}r{+}t$}{$r{-}t$}
\ladderEn{0}{1}{$k{-}s{+}r$}{$l{+}s{-}r$}{$s{-}t$}
\end{tikzpicture}
\label{eq:EF=FE}
\displaybreak[1]
\end{align}
\begin{equation}
\label{eq:REP4}
\renewcommand{\ladderY}{1}
\begin{ladder}{2}{3}
\ladderF{0}{0}{}{}
\ladderF{0}{1}{}{}
\ladderF{1}{2}{}{}
\ladderI{0}{2}
\ladderIn{2}{0}{2}
\end{ladder}
- 2
\begin{ladder}{2}{3}
\ladderF{0}{0}{}{}
\ladderF{1}{1}{}{}
\ladderF{0}{2}{}{}
\ladderI{0}{1}
\ladderI{2}{0}
\ladderI{2}{2}
\end{ladder}
+
\begin{ladder}{2}{3}
\ladderF{1}{0}{}{}
\ladderF{0}{1}{}{}
\ladderF{0}{2}{}{}
\ladderI{0}{0}
\ladderIn{2}{1}{2}
\end{ladder}
= 0
\end{equation}
together with their mirror reflections in the $y$-axis. We interpret unlabeled horizontal strands as being labeled by 1, and unlabeled vertical strands as being labeled by arbitrary compatible labels. These diagrams are to be interpreted as having some number of vertical strands to the left and right.

Define a \textbf{$\gl_n$-ladder} to be an equivalence class of $n$-ladders, where the equivalence relation is given by equations (\ref{eq:IHlad}) - (\ref{eq:REP4}) and invariance under orientation-preserving planar isotopy. Given two sequences $\lambda$, $\mu$ of integers we say that a $\gl_n$-ladder \textbf{connects $\lambda$ to $\mu$} if the sequence of edges along the bottom (respectively top) of the ladder diagram is $\lambda$ (respectively $\mu$). 

By definition of $\gl_n$-ladders we have that $\dU \gl_n$ is isomorphic to the category with
\begin{itemize}
\item
Objects: Sequences $\lambda \in \Z^n$.
\item
Morphisms: Morphisms from $\lambda$ to $\mu$ are $\C$-linear combinations of $\gl_n$-ladders connecting $\lambda$ to $\mu$.
\end{itemize}

\subsection{Schur-Weyl Duality and the CKM Principle}\label{Snd}
Let $V:= \C^n$. 
The weights of $V^{\otimes d}$ are elements of the set 
\[
\Lambda (n,d) := \{ (\lambda_1 , \ldots , \lambda_n)  \in \Z^{n}_{\geq 0} | \sum_{i} \lambda_i = d
\}.
\] 
Let $\Lambda^{+} (n,d)$ be the set of $\gl_n$-dominant weights of $\Lambda (n,d)$.
Schur-Weyl Duality states that $V^{\otimes d}$ is a $( \gl_n , \C[S_d])$-bimodule with saturated multiplicity free decomposition
\[
V^{\otimes d} = \bigoplus_{
\lambda \in \Lambda^{+} (n,d)
} V^{\lambda} \otimes L^{\lambda}
\]
where the $V^{\lambda}$ are the irreducible $\gl_n$-modules of highest weight $\lambda$, and the $L^{\lambda}$ are the Specht modules.

\begin{prop}\label{scwe}
Let $\dot{\mathcal{S}} (n,d)$ be the full subcategory of $\Rep S_d$ whose objects are the $\gl_n$-weight spaces of $V^{\otimes d}$. Then $\dot{\mathcal{S}} (n,d)$ is isomorphic to the category defined
\begin{itemize}
\item
Objects: Sequences $\lambda \in \Lambda (n,d)$.
\item
Morphisms: Morphisms from $\lambda$ to $\mu$ are $\C$-linear combinations of $\gl_n$-ladders connecting $\lambda$ to $\mu$, and satisfying the relation
\begin{align}\label{queen}
\begin{tikzpicture}[baseline=20, scale=0.75]
\laddercoordinates{3}{1}
\node[below] at (l00) {$\nu_{1}$};
\node[below] at (l20) {$\nu_{n}$};
\ladderI{0}{0};
\node at ($(l10)+(0,1)$) {$\cdots$};
\ladderI{2}{0};
\end{tikzpicture} 
&= 0
&\text{if $\nu_i < 0$ for some $i$.}
\end{align}
\end{itemize}
\end{prop}

\begin{proof}
By the CKM principle $\dot{\mathcal{S}} (n,d)$ is isomorphic to $\dU^{V^{\otimes d}} \gl_n$. Hence it suffices to show that (\ref{queen}) is equivalent to the relation
\begin{align*}
\begin{tikzpicture}[baseline=20, scale=0.75]
\laddercoordinates{3}{1}
\node[below] at (l00) {$\nu_{1}$};
\node[below] at (l20) {$\nu_{n}$};
\ladderI{0}{0};
\node at ($(l10)+(0,1)$) {$\cdots$};
\ladderI{2}{0};
\end{tikzpicture} 
&= 0
&\text{if $(\nu_1 , \ldots , \nu_n) \notin \Lambda (n , d)$.}
\end{align*}
Indeed $\nu \notin \Lambda (n,d)$ if and only if either $\sum_i \nu_i \neq d$ or $\nu_i < 0$ for some $i$. However it is impossible for a $\gl_n$-ladder to connect a sequence $\lambda \in \Lambda (n,d)$ to a sequence $\nu \in \Z^n$ in which $\sum_{i} \nu_i \neq d$. The result follows.
\end{proof}

By Proposition \ref{scwe} there is a bifunctor 
$
\cdot ||\cdot : \dot{\mathcal{S}} (n,d) \times \dot{\mathcal{S}}(m,d') \rightarrow \dot{\mathcal{S}} (n+m,d+d') 
$
defined on objects by $((\C^n)_{\lambda}^{\otimes d},(\C^m)_{\mu}^{\otimes d'}) \mapsto (\C^{m+n})_{(\lambda, \mu)}^{\otimes d+d'}$, and on morphisms by horizontal juxtaposition of ladder diagrams. 

\begin{rem}
For bifunctoriality of $\cdot ||\cdot $ we need that $\gl_n$-ladders are invariant under orientation preserving planar isotopy.
\end{rem}

\begin{prop}\label{bjm}
The bifunctor $\cdot ||\cdot$ is equal to the circle product bifunctor 
$\cdot \circ \cdot = \Ind_{S_d \times S_{d'}}^{S_{d+d'}}  (\cdot \boxtimes \cdot)
: \dot{\mathcal{S}} (n,d) \times \dot{\mathcal{S}}(m,d') \rightarrow \dot{\mathcal{S}} (n+m,d+d') 
$.
\end{prop}

\begin{proof}
It suffices to show that the bifunctors $\cdot ||\cdot$ and $\cdot \circ \cdot$ are equal on objects and generating morphisms.
Let $\lambda \in \Lambda (n,d)$ and $\mu \in \Lambda (m, d')$. 
Write $E_{i}^{(r)} 1_{\lambda} : (\C^n)_{\lambda}^{\otimes d} \rightarrow (\C^n)_{\lambda + \alpha_i}^{\otimes d}$ for the $S_d$-equivariant morphism $v \mapsto \frac{1}{r!}E_{i}^{r} v$, and $1_{\lambda}$ for the identity on $(\C^n)_{\lambda}^{\otimes d}$.
It suffices to show that
\begin{align}
\label{tull}
E_{i}^{(r)} 1_{\lambda} \circ 1_{\mu} &=
E_{i}^{(r)} 1_{(\lambda, \mu)} : (\C^{m+n})_{(\lambda, \mu)}^{\otimes d+d'} \rightarrow (\C^{m+n})_{(\lambda + \alpha_i, \mu)}^{\otimes d+d'}
\displaybreak[1]\\
1_{\lambda} \circ E_{i}^{(r)} 1_{\mu} &=
E_{i+n}^{(r)} 1_{(\lambda, \mu)} : (\C^{m+n})_{(\lambda, \mu)}^{\otimes d+d'} \rightarrow (\C^{m+n})_{(\lambda, \mu + \alpha_i)}^{\otimes d+d'}
\end{align}
and likewise with $F_i$ replacing $E_i$. We just show (\ref{tull}).

Let $\{ v_1 , \ldots , v_{n} \}$ and $\{ w_1 , \ldots w_{m} \}$ be the standard basis of $\C^{n}$ and $\C^{m}$ respectively. Write $\{ v_1 , \ldots , v_{n}, w_1 , \ldots w_{m} \}$ for the standard basis of 
$\C^{m+n}$.
There is a canonical inclusion of $\C [S_d] \times \C [S_{d'}]$-modules: 
\[
\iota: (\C^n)_{\lambda}^{\otimes d} \boxtimes (\C^m)_{\mu}^{\otimes d'} \rightarrow (\C^{m+n})_{(\lambda, \mu)}^{\otimes d+d'} ; v \otimes w \mapsto v \otimes w
\]
Define the $\C [S_d] \times \C [S_{d'}]$-equivariant morphism:
\[
(E_{i}^{(r)} 1_{\lambda} , 1_{\mu}): (\C^n)_{\lambda}^{\otimes d} \boxtimes (\C^m)_{\mu}^{\otimes d'} \rightarrow (\C^{m+n})_{(\lambda + \alpha_i, \mu)}^{\otimes d+d'} ; v \otimes w \mapsto E_{i}^{(r)} 1_{\lambda} (v) \otimes w
\]
Then $E_{i}^{(r)} 1_{(\lambda, \mu)}$ is the unique $\C[S_{d+d'}]$-equivariant morphism making the following diagram commute
\[
\xymatrix{
(\C^{m+n})_{(\lambda, \mu)}^{\otimes d+d'} \ar[rr]^{E_{i}^{(r)} 1_{(\lambda, \mu)}} && 
(\C^{m+n})_{(\lambda + \alpha_i, \mu)}^{\otimes d+d'}
 \\
(\C^n)_{\lambda}^{\otimes d} \boxtimes (\C^m)_{\mu}^{\otimes d'}
\ar[u]^{\iota} 
\ar[urr]_{(E_{i}^{(r)} 1_{\lambda}, 1_{\mu})}
}
\]
Equation (\ref{tull}) follows.

\end{proof}

\begin{rem}
Note that the Schur Algebra
$
S (n,d) := \End_{\C[S_d]} V^{\otimes d} \cong  \bigoplus_{\lambda , \mu \in \Lambda (n,d)} 
\Hom_{\dot{\mathcal{S}} (n,d)} (1_{\lambda}, 1_{\mu})
$.
Hence a presentation of $S(n,d)$ is given by generators $1_{\lambda}$, $\lambda \in \Z^n$, 
and $E_i , F_i$, $1 \leq i \leq n$. These satisfy relations  (\ref{rel:1}), (\ref{rel:2}), (\ref{rel:3}), (\ref{rel:4}) and 
\begin{align*}
1_{\lambda} &= 0 &\text{if $\lambda \notin \Lambda (n,d)$} \displaybreak[1]\\
1_{\lambda} 1_{\mu} &= \delta_{\lambda, \mu } 1_{\lambda} ,
\qquad \sum_{\lambda \in \Z^n} 1_{\lambda} = 1 \displaybreak[1]\\
E_i 1_{\lambda} &= 1_{\lambda + \alpha_i} E_i,
\qquad
F_i 1_{\lambda} = 1_{\lambda - \alpha_i} F_i
\end{align*}
This is the same presentation of $S (n,d)$ derived in \cite{doty}.
\end{rem}

\subsection{Proof of Theorem \ref{equiv}}\label{proof}
We now show that the functor $\Gamma: \Sp (\Perm) \rightarrow \Perm$ defined in Theorem \ref{equiv} is well defined and an isomorphism of categories. Our proof is essentially the same argument used in \cite{mainpaper} to give a diagrammatic presentation of the category of fundamental representations of $\sl_n$. 

We begin by defining the functor  $\Theta^{n}_d :  \dot{\mathcal{S}} (n,d) \rightarrow \Sp(\Perm)$
\begin{itemize}
\item
Objects: An object $V_{\lambda}^{\otimes d}$ is mapped to the composition, $\kappa(\lambda) \vDash d$ obtained from $\lambda$ by deleting 0 terms.
\item
Morphisms: $\gl_n$-ladders are identified with the $\Perm$-spiders drawn the same way.
\end{itemize}

To see that $\Theta^{n}_d$ is well defined we must check that equations (\ref{eq:IHlad})-(\ref{eq:REP4}) and (\ref{queen}) hold in $\Sp(Perm)$. It is immediate that (\ref{queen}) holds in $\Sp(\Perm)$. Equations (\ref{eq:IHlad}) and (\ref{eq:IHlad2}) hold in the category $\Sp (\Perm)$ by (\ref{eq2:IH}) and (\ref{eq2:HI}). Equations (\ref{eq:EE}) and (\ref{eq:EF=FE}) are the same as equations and (\ref{band}) and (\ref{eq2:commute}). Furthermore (\ref{eq:REP4}) holds in $\Sp (\Perm)$ -- see (\cite{mainpaper}, pp.8-9) for a proof.  

For any sequence $\lambda \in \Z^{n}$, let $l(\lambda)$ denote the length of $\lambda$.

\begin{lem}\label{l1}
The category $\dot{\mathcal{S}} (n,d)$ is equivalent to the full subcategory of $\Perm$ whose objects are the permutation modules $M^{\lambda}_{d}$ in which $l(\lambda) \leq n$.
Moreover the resulting fully faithful functors $\Psi^{n}_d : \dot{\mathcal{S}} (n,d) \rightarrow \Perm$ factor through $\Theta^{n}_d:  \dot{\mathcal{S}} (n,d) \rightarrow \Sp(\Perm)$.
\end{lem}

\begin{proof}
Let $\{ v_1, \ldots , v_n\}$ denote the standard basis of $\C^n$. 
For $\lambda \in \Lambda (n,d)$, the $S_d$-orbit of 
\[
v_{\lambda} := \underbrace{v_1 \otimes \cdots \otimes v_1}_\textrm{$\lambda_1$ times} \otimes \underbrace{v_2 \otimes \cdots \otimes v_2}_\textrm{$\lambda_2$ times} \otimes \cdots \otimes \underbrace{v_n \otimes \cdots \otimes v_n}_\textrm{$\lambda_n$ times}
\]
is a basis of the $\lambda$-weight space, $V_{\lambda}^{\otimes d}$, of $V^{\otimes d}$.
By the orbit-stabiliser relation, $V_{\lambda}^{\otimes d}$ is isomorphic, in $\Rep S_d$, to the permutation module $\C [S_{\lambda} \backslash S_d] = M^{\kappa(\lambda)}$. The first result follows.

We now show that $\Psi^{n}_d$ factors through $\Theta^{n}_d$. It suffices to show that, for any $\lambda, \mu \in \Lambda (n,d)$, the kernel of the linear map
\[
\Hom_{\dot{\mathcal{S}} (n,d)} (V^{\otimes d}_{\lambda}, V^{\otimes d}_{\mu})
\rightarrow
\Hom_{\Sp (\Perm)} (\kappa(\lambda) , \kappa(\mu))
\]
defined by $\Theta^{n}_d$ is contained in the kernel of the map
\[
\Hom_{\dot{\mathcal{S}} (n,d)} (V^{\otimes d}_{\lambda}, V^{\otimes d}_{\mu})
\rightarrow
\Hom_{\Perm} (M^{\kappa (\lambda)}, M^{\kappa (\mu)})
\]
defined by $\Psi^{n}_d$.
Since $\Psi^{n}_d$ is faithful, it suffices to show that $\Theta^{n}_d$ is faithful
i.e. the equations (\ref{eq2:IH}) - (\ref{eq2:commute}) hold in 
$\dot{\mathcal{S}} (n,d)$. The only equation that is not immediate is (\ref{eq2:bigon1}), which is a special case of (\ref{eq:EE}) with $l=k-s-r=0$.
\end{proof}

\begin{rem}
Let $\phi_{\lambda} : V^{\otimes d}_{\lambda} \rightarrow M^{\kappa(\lambda)}$ be the isomorphism in $\Rep S_d$ defined $v_{\lambda} \cdot g \mapsto S_{\kappa(\lambda)} g$ for all $g \in S_d$.
Then $\Psi^{n}_d : \dot{\mathcal{S}} (n,d) \rightarrow \Perm$ is defined explicitly
\begin{itemize}
\item
Objects: An object $V_{\lambda}^{\otimes d}$ is mapped to $M^{\kappa(\lambda)}$.
\item
Morphisms: A morphism $f: V_{\lambda}^{\otimes d} \rightarrow V_{\mu}^{\otimes d}$ in $\dot{\mathcal{S}} (n,d)$ is mapped to the morphism 
$\phi_{\mu} f \phi_{\lambda}^{-1}: M^{\kappa(\lambda)} \rightarrow M^{\kappa(\mu)}$ in $\Perm$.
\end{itemize}
\end{rem}

\begin{lem}\label{equivlem2}
The functor $\Gamma: \Sp (\Perm) \rightarrow \Perm$ is well defined and is the unique monoidal functor making the following diagram commute for all $n,d$:
\begin{equation}\label{prfprm}
\xymatrix{
 \dot{\mathcal{S}} (n,d) \ar[r]^{\Theta^{n}_d} \ar[dr]_{\Psi^{n}_d} & \Sp (\Perm) \ar[d]^{\Gamma}
 \\
 & \Perm
}
\end{equation}
\end{lem}

We first show that Theorem \ref{equiv} follows from Lemma \ref{equivlem2}. 

Each morphism in $\Perm$ is of the form $\Psi^{n}_d (f)$ for some $n,d$. Hence (\ref{prfprm}) implies that $\Gamma$ is full. A special case of
Theorem 5.3.1 in 
\cite{mainpaper}
is that every $\Perm$-spider is equal, in $\Sp (\Perm)$, to a $\gl_n$-ladder for some $n,d$. Hence, every morphism in $\Sp (\Perm)$ is of the form 
$\Theta^{n}_d (f)$ for some $n,d$. Since each $\Psi^{n}_d$ is faithful, this and (\ref{prfprm}) imply that if $\Gamma (f) =0$ then $f=0$. That is, $\Gamma$ is faithful. Finally it is clear that $\Gamma$ is bijective on objects. Hence Lemma \ref{equivlem2} suffices to show that $\Gamma: \Sp (\Perm) \rightarrow \Perm$ is an isomorphism of categories.

To prove Lemma \ref{equivlem2} we require the following sublemma.

\begin{lem}\label{l2}
Any functor $Sp (\Perm) \rightarrow \Perm$ making
\begin{equation}\label{prfprm2}
\xymatrix{
 \dot{\mathcal{S}} (2,k+l) \ar[r]^{\Theta^{2}_{k+l}} \ar[dr]_{\Psi^{2}_{k+l}} & \Sp (\Perm) \ar[d]^{}
 \\
 & \Perm
}
\end{equation}
commute sends 
\[
\fuse{k}{l}{k+l} \mapsto \nabla_{k,l}: M^{k,l} \rightarrow M^{k+l} \ \ \text{ , } \ \ \fork{k}{l}{k+l} \mapsto \Delta_{k,l}: M^{k+l} \rightarrow M^{k,l}
\]
\end{lem}

\begin{proof}
Now,
\begin{align*}
\fuse{k}{l}{k+l} = 
\begin{tikzpicture}[baseline=20]
\laddercoordinates{3}{1}
\node[below] at (l10) {$k$};
\node[below] at (l20) {$l$};
\node[above] at (l11) {$k{+}l$};
\node[above] at (l21) {$0$};
\ladderFn{1}{0}{}{}{$l$}
\end{tikzpicture} 
\qquad
\text{ and }
\qquad
\fork{k}{l}{k+l} = 
\begin{tikzpicture}[baseline=20]
\laddercoordinates{3}{1}
\node[below] at (l10) {$0$};
\node[below] at (l20) {$k{+}l$};
\node[above] at (l11) {$k$};
\node[above] at (l21) {$l$};
\ladderFn{1}{0}{}{}{$k$}
\end{tikzpicture} 
\end{align*}
in $\Sp (\Perm)$. 
Hence it suffices to show that
\begin{align}
\label{g1}
\Psi^{2}_{k+l} (E^{(l)}_1 1_{k,l}) 
&:M^{k,l} \rightarrow M^{k+l}; S_{k,l} g \mapsto S_{k+l}
\displaybreak[1]\\
\label{g2}
\Psi^{2}_{k+l} (E^{(k)}_1 1_{0,k+l})
&:
M^{k+l} \rightarrow M^{k,l};
S_{k+l} \mapsto \frac{1}{k!l!}
\sum_{g \in S_{k+l}} S_{k,l}g
\end{align}
Let $V=\C^{2}$. Description (\ref{g1}) holds since 
$E^{(l)}_1 1_{k,l}: V_{k,l}^{\otimes k{+}l} \rightarrow V_{k{+}l,0}^{\otimes k{+}l}$ maps
\[
 (\underbrace{v_1 \otimes \cdots \otimes v_1}_\textrm{$k$ times} \otimes \underbrace{v_2 \otimes \cdots \otimes v_2}_\textrm{$l$ times} ) \cdot g
\mapsto
v_1 \otimes \cdots \otimes v_1
\]
for all $g \in S_{k{+}l}$.
Description (\ref{g2}) holds since $E^{(k)}_1 1_{0,k+l}: V_{0,k{+}l}^{\otimes k{+}l} \rightarrow V_{k,l}^{\otimes k{+}l}$ maps 
\begin{align*}
v_2 \otimes \cdots \otimes v_2 &\mapsto
\sum_{\substack{R \subseteq \{1 , \ldots , k{+}l \} \\ |R|=k}}
v_{i_R (1)} \otimes \cdots \otimes v_{i_R (d)},
\;\;\;\;\;\;\;\;\;\;\;\;\;\;\; \text{where }
i_R (j) =
\begin{cases}
1 &\text{if $j \in R$,} \\
2 &\text{otherwise} 
\end{cases}
\displaybreak[1]\\
&=
\frac{1}{k!l!} \sum_{g \in S_{k{+}l}}
 (\underbrace{v_1 \otimes \cdots \otimes v_1}_\textrm{$k$ times} \otimes \underbrace{v_2 \otimes \cdots \otimes v_2}_\textrm{$l$ times} ) \cdot g
\end{align*}
This completes the proof of Lemma \ref{l2}.
\end{proof}

By Lemma \ref{l1} there is a unique functor $\Theta^{n}_d (\dot{\mathcal{S}} (n,d)) \rightarrow \Perm$
making the diagram
\begin{equation*}
\xymatrix{
 \dot{\mathcal{S}} (n,d) \ar[r]^{\Theta^{n}_d} \ar[dr]_{\Psi^{n}_d} & \Theta^{n}_d (\dot{\mathcal{S}} (n,d)) \ar[d]
 \\
 & \Perm
}
\end{equation*}
commute. By Lemma \ref{l2} and since the bifunctor $
\dot{\mathcal{S}} (n,d) \times \dot{\mathcal{S}}(m,d') \rightarrow \dot{\mathcal{S}} (n+m,d+d') 
$
defined by horizontal juxtaposition of diagrams
agrees with the circle product, this unique functor is $\Gamma |_{\Theta^{n}_d (\dot{\mathcal{S}} (n,d))} :\Theta^{n}_d (\dot{\mathcal{S}} (n,d)) \rightarrow \Perm$. Lemma \ref{equivlem2} follows since every morphism in $\Sp (\Perm)$ is a morphism in
$\Theta^{n}_d (\dot{\mathcal{S}} (n,d))$ for some $n,d$ (\cite{mainpaper}, Theorem 5.3.1).

This completes the proof of Theorem \ref{equiv}.

\section{A Tabloid-Theoretic Description of the Generating Morphisms of $\dot{\mathcal{S}} (n,d)$}\label{Snd1}
By Proposition \ref{scwe} there are morphisms
 \begin{align*}
E^{(r)}_i 1_{\lambda} &: M^{\lambda} \rightarrow M^{\lambda+ r\alpha_i}  \\
F^{(r)}_i 1_{\lambda} &: M^{\lambda} \rightarrow M^{\lambda- r\alpha_i}
\end{align*}
that generate all $S_d$-equivariant morphisms between the permutation modules $M^{\lambda}$ for $\lambda \in \Lambda (n,d)$. It will be useful in Sections \ref{prm} and \ref{kronecker} to have an explicit description for these maps. We derive this now.

For $m \in \N$, define $\underline{m} := \{ 1 , \ldots , m \}$.
The basis elements $v= v_{i_1} \otimes \cdots \otimes v_{i_d} \in V^{\otimes d} := (\C^n)^{\otimes d}$ correspond precisely to dissections, $T_v$, of $\underline{d}$ into $n$ subsets,
\begin{align*}
T_v = \{ T^{1}_v , \ldots , T^{n}_v\}
\qquad
\text{where }
T^{j}_v = \{ k \in \underline{d} | i_k = j \}.
\end{align*}
If $v= v_{i_1} \otimes \cdots \otimes v_{i_d} \in V^{\otimes d}$ has weight $\lambda$, we call such a dissection a \textbf{$\lambda$-tabloid} (or \textbf{tabloid} if $\lambda$ is unspecified).
We depict such a tabloid by an $n \times 1$ array of cells, indexed vertically downward, whose $j$-th cell contains the elements of $T^{j}_v$. For example,
\begin{align*}
T_{v_1 \otimes v_1 \otimes v_1 \otimes v_2 \otimes v_4}
&=
\begin{tabular}{| c | }
\hline 
 1 , 2 , 3 \\
\hline
4 \\
\hline
 \\
\hline
5 \\
\hline
\end{tabular}
\end{align*}

Given a tabloid $T$, write $T^i$ for the set of elements in the $i$-th cell of $T$. 
Define $Y^{\lambda}$ to be the $\C$-span of $\lambda$-tabloids. 
Now, $Y^{\lambda}$ is a right $\C[S_d]$-module with the action $T_{v} \cdot g = T_{vg}$. More precisely,
\begin{align*}
(Tg)^{i} = g^{-1} (T^i)
\qquad
\text{for all tabloids }
T \in Y^{\lambda}.
\end{align*}

By definition $Y^{\lambda}$ is isomorphic, in $\Rep S_d$, to $M^{\lambda}$. 
More precisely, 
for $\lambda \in \Lambda (n,d)$, define
\[
T_{\lambda} := 
\begin{tabular}{| c | }
\hline 
 $1 , \ldots , \lambda_1$ \\
\hline
$\lambda_1 + 1 , \ldots , \lambda_1 + \lambda_2 $\\
\hline
$ \vdots$\\
\hline
$\lambda_1 + \cdots + \lambda_{n-1} +1 , \ldots , d$\\
\hline
\end{tabular}
\]
There is an isomorphism $\phi_{\lambda}:  Y^{\lambda} \rightarrow M^{\lambda}$ defined $T_{\lambda} \cdot g \mapsto S_{\lambda} g$. \emph{Henceforth identify $M^{\lambda}$ with $Y^{\lambda}$ via the isomorphism $\phi_{\lambda} $}.

\begin{prop}\label{tabact}
Let $R \subseteq \underline{d}$, $i \in \underline{n}$. For any tabloid $T$ define $c_{i,R} T$ to be the tabloid obtained from $T$ by moving every element in the set $R$ to the $i$-th cell. 

The maps 
$
E^{(r)}_i 1_{\lambda} : M^{\lambda} \rightarrow M^{\lambda+ r\alpha_i}  
$
and
$
F^{(r)}_i 1_{\lambda} : M^{\lambda} \rightarrow M^{\lambda- r\alpha_i}
$
are defined on tabloids by,
 \begin{align*}
E^{(r)}_i 1_{\lambda} &: T \mapsto \sum_{R \subseteq T^{i+1}, |R| = r} c_{i,R} T\\
F^{(r)}_i 1_{\lambda} &: T \mapsto \sum_{R \subseteq T^{i}, |R| = r} c_{i+1,R} T
\end{align*}
\end{prop}

\begin{proof}
The left action
$
\gl_n \curvearrowright \bigotimes^d \C^n = \bigoplus_{\lambda \in \Lambda(n,d)} M^{\lambda}
$
is defined on tabloids by the rule $x \cdot T_v = T_{x \cdot v}$ (here we interpret an expression $T_{v+w}$ as $T_{v}+T_{w}$).
We first show that, for any tabloid $T$,
\begin{equation*}
e_{ij} T = \sum_{k \in T^j} c_{i,\{k\}} T
\end{equation*}
Indeed, given $v = v_{i_1} \otimes \cdots \otimes v_{i_d}$,
\begin{align*}
e_{ij} v &=  \sum_{k \in T^{j}_v} v_{i_1} \otimes \cdots \otimes v_{i_{k-1}} \otimes v_i \otimes v_{i_{k+1}} \otimes \cdots \otimes v_{i_d}
\\
\intertext{Hence, }
e_{ij}T_v &= T_{e_{ij} v} = \sum_{k \in T^{j}_v} c_{i,\{k\}} T_v
\end{align*}
as required. Moreover,
\begin{equation*}
e_{ij}^{(r)} T = \sum_{R \subseteq T^{j}, |R| = r} c_{i,R} T
\end{equation*}
The result follows.
\end{proof}

For example, 
\begin{align*}
F_1 1_{(2,1)} \cdot
\begin{tabular}{| c | }
\hline 
 1 , 2 \\
\hline
3 \\
\hline
\end{tabular}
&=
e_{21} \cdot T_{v_1 \otimes v_1 \otimes v_2}
=
T_{v_2 \otimes v_1 \otimes v_2} + T_{v_1 \otimes v_2 \otimes v_2}
\displaybreak[1]\\
&=
\begin{tabular}{| c | }
\hline 
2 \\
\hline
1,3 \\
\hline
\end{tabular}
+
\begin{tabular}{| c | }
\hline 
 1 \\
\hline
2,3 \\
\hline
\end{tabular}
=
c_{2, \{1\}}
\begin{tabular}{| c | }
\hline 
 1 , 2 \\
\hline
3 \\
\hline
\end{tabular}
+
c_{2,\{2\}}
\begin{tabular}{| c | }
\hline 
 1 , 2 \\
\hline
3 \\
\hline
\end{tabular}
\end{align*}

\begin{rem}
The maps $\nabla_{k,l}: M^{k,l} \rightarrow M^{k+l}$ and $\Delta_{k,l}:M^{k+l} \rightarrow M^{k,l}$ are defined on tabloids
\begin{align*}
\nabla_{k,l} \left( 
\begin{tabular}{| c | }
\hline 
 1 , \ldots , $k $\\
\hline
$k+1$ , \ldots , $k+l$ \\
\hline
\end{tabular}
\cdot
g
\right)
&=
\begin{tabular}{| c | }
\hline 
 1 , \ldots ,$ k+l$ \\
\hline
\end{tabular}
\displaybreak[1]\\
\Delta_{k,l}
\left( 
\begin{tabular}{| c | }
\hline 
 1 , \ldots ,$ k+l$ \\
\hline
\end{tabular}
\right)
&=
\frac{1}{k!l!}
\sum_{g \in S_{k+l}}
\begin{tabular}{| c | }
\hline 
 1 , \ldots , $k $\\
\hline
$k+1$ , \ldots , $k+l$ \\
\hline
\end{tabular}
\cdot
g
\end{align*}
\end{rem}

The circle product on permutation modules can be thought of as `vertical juxtaposition of tabloids'. More precisely, 
fix $\lambda, \lambda' \in \Lambda(n,d)$ and $\mu, \mu' \in \Lambda(m,d')$. Given a $\lambda$-tabloid $T$ and $\mu$-tabloid $T'$ define $T \circ T'$ to be the $(\lambda,\mu)$-tabloid defined:
\begin{itemize}
\item
For $1 \leq i \leq n$, $(T \circ T')^i = T^i$
\item
For $n < i \leq n+m$, $(T \circ T')^i = \{ k+d | k \in T'^{i-n} \}$
\end{itemize}
For example,
\begin{tabular}{| c | }
\hline 
2 , 3 \\
\hline
1 , 4 \\
\hline
 \\
\hline
\end{tabular}
$\circ$
\begin{tabular}{| c | }
\hline 
 1 , 2 \\
\hline
3 \\
\hline
\end{tabular}
$=$
\begin{tabular}{| c | }
\hline 
 2 , 3 \\
\hline
1 , 4 \\
\hline
 \\
\hline
5 , 6 \\
\hline
7 \\
\hline
\end{tabular}.

If $f: M^{\lambda} \rightarrow M^{\lambda'}$ and $g: M^{\mu} \rightarrow M^{\mu'}$ are $S_d$ and $S_d'$ equivariant respectively, then $f \circ g : M^{(\lambda, \mu)} \rightarrow M^{(\lambda', \mu')}$ is the $S_{d+d'}$-equivariant map defined on the generator $T_{(\lambda,\mu)} = T_{\lambda} \circ T_{\mu}$:
\begin{itemize}
\item
For $1 \leq i \leq n$, 
\[
((f \circ g)(T_{\lambda, \mu}))^i = (f(T_{\lambda}))^i
\]
\item
For $n < i \leq n+m$, 
\[
((f \circ g)(T_{\lambda , \mu}))^i = \{ k+ d | k \in (g(T_{\mu}))^{i-n} \}
\]
\end{itemize}

For example,
$1_{M^{(1,1)}} \circ F_{1} 1_{(2,1)} : M^{(1,1,2,1)} \rightarrow M^{(1,1,1,2)}$ is defined, for $g \in S_d$,
\[
1_{M^{(1,1)}} \circ F_{1} 1_{(2,1)} : 
\begin{tabular}{| c | }
\hline 
1 \\
\hline
2 \\
\hline
3,4 \\
\hline
5 \\
\hline
\end{tabular}
\cdot g
\mapsto
\left(
\begin{tabular}{| c | }
\hline 
1 \\
\hline
2 \\
\hline
4 \\
\hline
3,5 \\
\hline
\end{tabular}
+
\begin{tabular}{| c | }
\hline 
1 \\
\hline
2 \\
\hline
3 \\
\hline
4,5 \\
\hline
\end{tabular}
\right)
\cdot g
\]
Notice from this example that $1_{M^{(1,1)}} \circ F_{1} 1_{(2,1)} = F_3 1_{(1, 1, 2,1)}$ as expected by our identification of the circle product with horizontal juxtaposition of ladder diagrams.

\section{Symmetric Monoidal Structure on $\Perm$}\label{prm}

For $\lambda \in \Z^n$, $d \in \Z$, write $\lambda \vdash d$ if $\lambda$ is a partition of $d$. In this section we describe the braiding maps for the $\circ$-product in $\Perm$ (Theorem \ref{permsym}), and
derive a diagrammatic presentation of the full (skeletal) subcategory, $\Prm$, of $\Perm$ whose objects are the permutation modules, $M^{\lambda}$, for $\lambda \vdash d$, $d \in \N$ (Corollary \ref{Prmweb}).

\begin{thm}\label{permsym}
The category $\Perm$ is symmetric monoidal with braiding isomorphisms defined diagrammatically,
\begin{align*}
\tikz[baseline=40]{
\draw[mid>] (0,0) 
to [out=up, in=down] (2,3);
\draw[mid>] (2,0) 
to [out=up, in=down] (0,3);
\node[left] at (0,0) {$k$};
\node[right] at (2,0) {$l$};
\node[left] at (0,3) {$l$};
\node[right] at (2,3) {$k$};
}
&:=
\sum_{a,b \geq 0 , a-b=k-l} (-1)^{k-a} 
\begin{tikzpicture}[baseline=40]
\laddercoordinates{1}{2}
\node[left] at (l00) {$k$};
\node[right] at (l10) {$l$};
\ladderEn{0}{0}{$k{-}a$}{$l{+}a$}{a}
\ladderFn{0}{1}{$l$}{$k$}{b}
\end{tikzpicture}
\\
&=
\sum_{a,b \geq 0 , a-b=l-k} (-1)^{l-a} 
\begin{tikzpicture}[baseline=40]
\laddercoordinates{1}{2}
\node[left] at (l00) {$k$};
\node[right] at (l10) {$l$};
\ladderFn{0}{0}{$k{+}a$}{$l{-}a$}{a}
\ladderEn{0}{1}{$l$}{$k$}{b}
\end{tikzpicture}
\end{align*}
\end{thm}

\begin{rem}
The sums in Theorem \ref{permsym} are finite since we interpret strands labelled by non-positive integers as zero morphisms.
\end{rem}

\begin{proof}
Let $B_{k,l}: M^{k,l} \rightarrow M^{l,k}$ be the isomorphism in $\Perm$ defined on $(k,l)$-tabloids by interchanging elements in the 1st and 2nd cells. That is,
\[
B_{k,l} (T) = c_{2,T^1} c_{1,T^2} T
\]
It is clear that the $B_{k,l}$ define braidings for $\Perm$ that give $\Perm$ the structure of a symmetric monoidal category.
Let
\[
B^{k,l} := \sum_{a,b \geq 0, a-b=k-l} (-1)^{k-a} E_{1}^{(b)} F_{1}^{(a)} 1_{(k,l)} : M^{k,l} \rightarrow M^{l,k}
\]
We first show that $B_{k,l} = B^{k,l}$.

By Proposition \ref{tabact}, for any $(k,l)$-tabloid $T$,
\[
 E_{1}^{(b)} F_{1}^{(a)} 1_{(k,l)} (T) =  \sum_{R \subseteq T^{1}, |R| = a} \sum_{S \subseteq T^{2} \cup R, |S| = b} c_{1,S} c_{2,R} T
\]
Fix $R \subsetneq T^1$ of size $a$, and $S \subsetneq T^{2}$ of size $b:= a - k + l$. Then 
\begin{align*}
c_{1,S} c_{2,R} T= c_{1, S \cup Q} c_{2, R \cup Q} T
\qquad
\text{for all $Q \subseteq T^1 \backslash R$.}
\end{align*}
In particular, $c_{1,S} c_{2,R} T$ appears $\binom{k - a }{m} = \binom{k-a}{k-(a+m)}$ times as a summand of $E_{1}^{(b+m)} F_{1}^{(a+m)} (T)$ ; once for each $Q \subseteq T^1 \backslash R$ such that $|Q|=m$.  Hence the number of times  $c_{1,S} c_{2,R}T$ appears as a summand of $B^{k,l} (T)$  is 
\[
\sum_{\text{$k-(a+m)$ even}} \binom{k- a}{k-(a+m)} - \sum_{\text{$k-(a+m)$ odd}} \binom{k - a}{k-(a+m)} =0
\]
Furthermore, $c_{1, T^{2}} c_{2, T^{1}} T$ appears exactly once in $B^{k,l}(T)$ as a summand of $E_{1}^{(l)} F_{1}^{(k)} T$. Hence,
\[
B^{k,l} (T) = c_{1, T^{2}} c_{2, T^{1}} T = B_{k,l} (T)
\]
as required.

By interchanging the roles of $T^1$ and $T^2$ in the above proof we also get that
\[
B_{k,l} = \sum_{a,b \geq 0, a-b=l-k} (-1)^{l-a} F_{1}^{(b)} E_{1}^{(a)} 1_{(k,l)} : M^{k,l} \rightarrow M^{l,k}
\]
This completes the proof of Theorem \ref{permsym}.
\end{proof}

We note the following feature of the diagram calculus for $\Perm$.
\begin{prop}
The following relations hold in $\Perm$:
\begin{align}
\tikz[baseline=40, yscale=0.5]{
\draw[mid>] (0.5,0) 
to [out=up, in=225] (1,2);
\draw[mid>] (1.5,0) 
to [out=up, in=-45] (1,2);
\draw[mid>] (1,2) 
to [out=up, in=down] (1,4);
\draw[mid>] (2.5,0) 
to [out=up, in=down] (0,2);
\draw[mid>] (0,2) 
to [out=up, in=down] (0,4);
\node[left] at (0.5,0) {$k$};
\node[right] at (1.5,0) {$l$};
\node[right] at (1,4) {$k{+}l$};
\node[right] at (2.5,0) {$m$};
\node[left] at (0, 4) {$m$}
}
&=
\tikz[baseline=40, yscale=0.5]{
\draw[mid>] (0.5,0) 
to [out=up, in=225] (1,2);
\draw[mid>] (1.5,0) 
to [out=up, in=-45] (1,2);
\draw[mid>] (1,2) 
to [out=up, in=down] (1,4);
\draw[mid>] (2.5,0) 
to [out=up, in=down] (2.5,1);
\draw[mid>] (2.5,1) 
to [out=up, in=down] (0,4);
\node[left] at (0.5,0) {$k$};
\node[right] at (1.5,0) {$l$};
\node[right] at (1,4) {$k{+}l$};
\node[right] at (2.5,0) {$m$};
\node[left] at (0, 4) {$m$}
}
\label{cross1}
\displaybreak[1]\\
\label{cross2}
\tikz[baseline=40, yscale=0.5]{
\draw[mid<] (0.5,4) 
to [out=up, in=135] (1,2);
\draw[mid<] (1.5,4) 
to [out=up, in=45] (1,2);
\draw[mid<] (1,2) 
to [out=up, in=down] (1,0);
\draw[mid>] (2.5,0) 
to [out=up, in=down] (0,3);
\draw[mid>] (0,3) 
to [out=up, in=down] (0,4);
\node[right] at (0.5,4) {$k$};
\node[right] at (1.5,4) {$l$};
\node[right] at (1,0) {$k{+}l$};
\node[right] at (2.5,0) {$m$};
\node[left] at (0, 4) {$m$}
}
&=
\tikz[baseline=40, yscale=0.5]{
\draw[mid<] (0.5,4) 
to [out=up, in=135] (1,2);
\draw[mid<] (1.5,4) 
to [out=up, in=45] (1,2);
\draw[mid<] (1,2) 
to [out=up, in=down] (1,0);
\draw[mid>] (2.5,0) 
to [out=up, in=down] (2.5,1);
\draw[mid>] (2.5,1) 
to [out=up, in=down] (0,4);
\node[right] at (0.5,4) {$k$};
\node[right] at (1.5,4) {$l$};
\node[right] at (1,0) {$k{+}l$};
\node[right] at (2.5,0) {$m$};
\node[left] at (0, 4) {$m$}
}
\end{align}
\end{prop}

\begin{proof}
Both equations (\ref{cross1}) and (\ref{cross2}) can be checked by evaluating each side of the equation of tabloids. For example, both sides of equation (\ref{cross1}) act on tabloids in the following way:
\[
\begin{tabular}{| c | }
\hline 
 1 , \ldots , $k $\\
\hline
$k{+}1$ , \ldots , $k{+}l$ \\
\hline
$k{+}l{+}1$ , \ldots, $k{+}l{+}m$ \\
\hline
\end{tabular}
\cdot
g
\mapsto
\begin{tabular}{| c | }
\hline 
 $k{+}l{+}1$ , \ldots , $k{+}l{+}m $\\
\hline
$1$ , \ldots , $k{+}l$ \\
\hline
\end{tabular}
\cdot
g
\]
for all $g \in S_{k+l+m}$.
\end{proof}

\begin{rem}
With a little work one may also find diagrammatic proofs of (\ref{cross1}) and (\ref{cross2}). We do not do this here.
\end{rem}

\subsection{A diagrammatic presentation of $\Prm$}
Recall the symmetric group $S_n$ acts on $\Z^n$ by permuting entries. 
Given $\mu \vDash d$ and $\lambda \vdash d$, say that $\lambda$ \textbf{dominates} $\mu$ (write $\mu \prec \lambda$) if $\mu$ is in the $S_n$ orbit of $\lambda$. 

\begin{cor}\label{Prmweb}
The category $\Prm$ is isomorphic as a $\C$-linear monoidal category to the category with
\begin{itemize}
\item
Objects: Partitions $\lambda \vdash d$ for some $d \in \N$.
\item
Morphisms: Morphisms from $\lambda$ to $\mu$ are $\C$-linear combinations of $\Perm$-spiders connecting $\lambda'$ to $\mu'$, where $\lambda' \prec \lambda$ and $\mu' \prec \mu$. These morphisms satisfy relations (\ref{eq2:IH}), (\ref{eq2:HI}), (\ref{eq2:bigon1}), (\ref{eq2:commute}), as well as the relation
\begin{align}
\tikz[baseline=20,xscale=0.75, yscale=1.5]{
\node[below] at (l00) {$k$};
\node[below] at (l10) {$l$};
\ladderI{0}{0};
\ladderI{1}{0};
}
=
\tikz[baseline=20,scale=0.75]{
\draw[mid>] (0,0) 
to [out=up, in=down] (2,3);
\draw[mid>] (2,0) 
to [out=up, in=down] (0,3);
\node[left] at (0,0) {$k$};
\node[right] at (2,0) {$l$};
\node[left] at (0,3) {$l$};
\node[right] at (2,3) {$k$};
}
=
\tikz[baseline=20,xscale=0.75, yscale=1.5]{
\node[below] at (l00) {$l$};
\node[below] at (l10) {$k$};
\ladderI{0}{0};
\ladderI{1}{0};
}
\qquad
\text{if $k \neq l$.}
\end{align}
The composition of morphisms $f: \lambda \rightarrow \mu$ and $g: \mu \rightarrow \nu$ is given by vertical juxtaposition of a diagram, 
$\tilde{f}$, depicting $f$ above a diagram, $\tilde{g}$, depicting $g$, in which the sequence along the top of $\tilde{g}$ is the same as the sequence along the bottom of $\tilde{f}$.
The monoidal product is given by horizontal juxtaposition.
\end{itemize}
\end{cor}

\begin{proof}
Clearly two permutation modules $M^{\lambda}$ and $M^{\lambda'}$ are isomorphic in $\Perm$ if and only if $\lambda$ and $\lambda'$ are dominated by the same partition. Hence $\Prm$ is the skeleton of $\Perm$.
It is not hard to check that, given $\lambda' \prec \lambda$, there is a unique isomorphism $\varphi_{\lambda'}: M^{\lambda'} \rightarrow M^{\lambda}$, in $\Perm$, made up by using composition and the $\circ$-product on the braiding isomorphisms $B_{k,l}: M^{k,l} \rightarrow M^{l,k}$ where $k \neq l$.
As $\Prm$ is the skeleton of $\Perm$, $\Prm$ is isomorphic to the category with
\begin{itemize}
\item
Objects: Isomorphism classes of objects in $\Perm$. Given a partition $\lambda \vdash d$, let $[\lambda]$ denote the isomorphism class containing $M^{\lambda}_{d}$.
\item
Morphisms: Morphisms $[\lambda] \rightarrow [\mu]$ are equivalence classes of elements in the disjoint union
\[
\coprod_{\lambda' \prec \lambda , \mu' \prec \mu} \Hom_{\Perm} (M^{\lambda'}, M^{\mu'})
\]
where a morphism $f: M^{\lambda'} \rightarrow M^{\mu'}$ is equivalent to 
$\varphi_{\mu'} f \varphi_{\lambda'}^{-1}: M^{\lambda} \rightarrow M^{\mu}$ for all $\lambda' \prec \lambda$, $\mu' \prec \mu$.
\end{itemize}
By Theorem \ref{equiv} this simplifies to the diagrammatic presentation of $\Prm$ given in Corollary \ref{Prmweb}.
\end{proof}

We next prove a useful feature of the diagrammatic calculus for $\Prm$. 
\begin{prop}
The following relations hold in $\Prm$:
\begin{align}
\fuse{k}{l}{k+l}
=
\fuse{l}{k}{k+l}
\label{vertexreflection}
\displaybreak[1]
\\
\fork{k}{l}{k+l}
=
\fork{l}{k}{k+l}
\label{vertexreflection2}
\end{align}
\end{prop}

\begin{proof}
To show equation (\ref{vertexreflection}) it suffices to show that if $k \neq l$ then
\[
\tikz[baseline=20,yscale=0.5]{
\draw[mid>] (0.5,0) 
to [out=up, in=245] (1,3);
\draw[mid>] (1.5,0) 
to [out=up, in=-65] (1,3);
\draw[mid>] (1,3) 
to [out=up, in=down] (1,4);
\node[left] at (0.5,0) {$k$};
\node[right] at (1.5,0) {$l$};
\node[left] at (1,4) {$k{+}l$};
}
=
\tikz[baseline=20,yscale=0.5]{
\draw[mid>] (0.5,0) 
to [out=up, in=down] (1.5,2);
\draw[mid>] (1.5,2) 
to [out=up, in=-45] (1,3);
\draw[mid>] (1.5,0) 
to [out=up, in=down] (0.5,2);
\draw[mid>] (0.5,2) 
to [out=up, in=225] (1,3);
\draw[mid>] (1,3) 
to [out=up, in=down] (1,4);
\node[left] at (0.5,0) {$k$};
\node[right] at (1.5,0) {$l$};
\node[left] at (1,4) {$k{+}l$};
\node[right] at (1.5,2) {$k$};
\node[left] at (0.5,2) {$l$}
}
\]
This equation can easily be verified by evaluating each side on $(k,l)$-tabloids. Equation (\ref{vertexreflection2}) holds similarly.
\end{proof}

\section{The Monoidal Structure on $\Prm (S_d)$}\label{kronecker}

Define $\Prm (S_d)$ to be the full skeletal subcategory of $\Rep S_d$ whose objects are direct sums of permutation modules $M^{\lambda}$, where $\lambda \vdash d$. 

Given an object $\bigoplus_{i \in \chi} M_i$ in $\Prm (S_d)$, and $k \in \chi$, write $\iota_{k}: M_i \rightarrow \bigoplus_{i \in \chi} M_i$ for the canonical inclusion and $\rho_{k}: \bigoplus_{i \in \chi} M_i \rightarrow M_i$ for the canonical projection.
A morphism $f:\bigoplus_{j} M_j \rightarrow \bigoplus_{i} M_{i}$ in $\Prm (S_d)$ can be identified with a matrix $(f_{ij})_{ij}$ where $f_{ij} = \rho_i f \iota_j : M_{j} \rightarrow M_{i}$. 
We slightly abuse notation by writing
\[
f = (f_{ij})_{ij} = \sum_{i} \sum_{j} f_{ij}
\]

Given any $\lambda \in \Lambda(n,d)$ in which $\kappa(\lambda)$ is dominated by the partition $\mu \vdash d$, we may write $M^{\lambda}$ to mean the object $M^{\mu}$ in $\Prm (S_d)$. For example, $M^{(1,0,2)}$ and $M^{(2,1)}$ both denote the same object in $\Prm (S_3)$. For $\lambda \in \Lambda(n,d)$, define morphisms $E_i 1_{\lambda} : M^{\lambda} \rightarrow M^{\lambda + \alpha_i}$ and $F_i 1_{\lambda} : M^{\lambda} \rightarrow M^{\lambda - \alpha_i}$ in $\Prm (S_d)$ by the diagrams
\begin{align*}
E_i \one_{\lambda} &:=
\begin{tikzpicture}[baseline=20, scale=0.75]
\laddercoordinates{3}{1}
\node[below] at (l00) {$\lambda_{i{-}1}$};
\node[below] at (l10) {$\lambda_i$};
\node[below] at (l20) {$\lambda_{i{+}1}$};
\node[below] at (l30) {$\lambda_{i{+}2}$};
\node[above] at (l11) {$\lambda_i{+}1$};
\node[above] at (l21) {$\lambda_{i{+}1}{-}1$};
\node at ($(l00)+(-1,1)$) {$\cdots$};
\ladderI{0}{0};
\ladderFn{1}{0}{}{}{$1$}
\ladderI{3}{0};
\node at ($(l30)+(1,1)$) {$\cdots$};
\end{tikzpicture} 
\qquad
\text{and}
\qquad
F_i \one_{\lambda} &:=
\begin{tikzpicture}[baseline=20, scale=0.75]
\laddercoordinates{3}{1}
\node[below] at (l00) {$\lambda_{i{-}1}$};
\node[below] at (l10) {$\lambda_i$};
\node[below] at (l20) {$\lambda_{i{+}1}$};
\node[below] at (l30) {$\lambda_{i{+}2}$};
\node[above] at (l11) {$\lambda_i{-}1$};
\node[above] at (l21) {$\lambda_{i{+}1}{+}1$};
\node at ($(l00)+(-1,1)$) {$\cdots$};
\ladderI{0}{0};
\ladderEn{1}{0}{}{}{$1$}
\ladderI{3}{0};
\node at ($(l30)+(1,1)$) {$\cdots$};
\end{tikzpicture} 
\end{align*}

Let $\lambda, \mu \vdash d$ of length $m$, $n$, respectively. We define $A^{\lambda}_{\mu}$ to be the set of $m \times n$ matrices $A=(A_{ij})_{ij}$ with entries in $\N$, such that $\lambda_i = \sum_j A_{ij}$ and $\mu_j = \sum_{i} A_{ij}$. For example
\begin{align*}
A^{(3,1)}_{(2,2)}
&=
\{
\begin{bmatrix}
    2 & 1  \\
    0 & 1 \\
  \end{bmatrix}
,
\begin{bmatrix}
    1 & 2  \\
    1 & 0 \\
  \end{bmatrix}
\}
\text{,}
\qquad
A^{(3,1)}_{(3,1)}
=
\{
\begin{bmatrix}
    3 & 0  \\
    0 & 1 \\
  \end{bmatrix}
,
\begin{bmatrix}
    2 & 1  \\
    1 & 0 \\
  \end{bmatrix}
\}.
\end{align*}
Given a matrix $A \in A^{\lambda}_{\mu}$, let $M^{A}$ denote the permutation module indexed by the sequence 
$(A_{11}, A_{12}, \ldots , A_{21} , A_{22} , \ldots , A_{mn})$. For $1 \leq i \leq m$, $1 \leq j \leq n-1$, 
define morphisms $E_{ij} 1_A : M^A \rightarrow M^{A+e_{ij} - e_{i,j+1}}$ and 
$F_{ij} 1_A : M^A \rightarrow M^{A-e_{ij} + e_{i,j+1}}$ in $\Prm (S_d)$ by the diagrams,
\begin{align*}
E_{ij} \one_{A} : M^A \rightarrow M^{A+e_{ij} - e_{i,j+1}} &:=
\begin{tikzpicture}[baseline=20]
\laddercoordinates{6}{1}
\node[below] at (l00) {$A_{11}$};
\node[below] at (l10) {$\cdots$};
\node[below] at (l20) {$A_{ij}$};
\node[below] at (l30) {$A_{i,j+1}$};
\node[below] at (l40) {$\cdots$};
\node[below] at (l50) {$A_{mn}$};
\node[above] at (l21) {$A_{ij}{+}1$};
\node[above] at (l31) {$A_{i,j+1}{-}1$};
\node at ($(l10)+(-0.75,1)$) {$\cdots$};
\ladderI{0}{0};
\ladderI{1}{0};
\ladderFn{2}{0}{}{}{$1$}
\ladderI{4}{0};
\ladderI{5}{0};
\node at ($(l40)+(0.75,1)$) {$\cdots$};
\end{tikzpicture} 
\displaybreak[1]\\
F_{ij} \one_{A} : M^A \rightarrow M^{A-e_{ij} + e_{i,j+1}}&:=
\begin{tikzpicture}[baseline=20]
\laddercoordinates{6}{1}
\node[below] at (l00) {$A_{11}$};
\node[below] at (l10) {$\cdots$};
\node[below] at (l20) {$A_{ij}$};
\node[below] at (l30) {$A_{i,j+1}$};
\node[below] at (l40) {$\cdots$};
\node[below] at (l50) {$A_{mn}$};
\node[above] at (l21) {$A_{ij}{-}1$};
\node[above] at (l31) {$A_{i,j+1}{+}1$};
\node at ($(l10)+(-0.75,1)$) {$\cdots$};
\ladderI{0}{0};
\ladderI{1}{0};
\ladderEn{2}{0}{}{}{$1$}
\ladderI{4}{0};
\ladderI{5}{0};
\node at ($(l40)+(0.75,1)$) {$\cdots$};
\end{tikzpicture}  
\end{align*}

If $A \in A^{\lambda}_{\mu}$ then 
\begin{enumerate}
\item
Either $A + (e_{ij} - e_{i,j+1}) \in A^{\lambda}_{\mu + \alpha_j}$ or $E_{ij} 1_A = 0$. The latter case occurs if and only if $A_{i,j+1} =0$.
\item
Either $A - (e_{ij} - e_{i,j+1}) \in A^{\lambda}_{\mu - \alpha_j}$ or $F_{ij} 1_A = 0$. The latter case occurs if and only if $A_{ij} =0$.
\end{enumerate}
Hence the following theorem makes sense.

\begin{thm}\label{monprm}
The category $\Prm (S_d)$ is closed under $\otimes$-product. In particular the bifunctor 
$\cdot \otimes \cdot : \Prm (S_d) \times \Prm (S_d) \rightarrow \Prm (S_d)$ can be expressed,
\begin{itemize}
\item
Objects: 
\begin{equation*}
M^{\lambda} \otimes M^{\mu} = \bigoplus_{A \in A_{\mu}^{\lambda}} M^{A} 
= \bigoplus_{A \in A_{\lambda}^{\mu}} M^{A}
\end{equation*}
where a matrix $A$ is regarded as the sequence $(A_{11}, A_{12}, \ldots , A_{21} , A_{22} , \ldots , A_{mn})$. 
\item
Morphisms:
Let $\lambda, \mu \vdash d$ of length $m$, $n$, respectively. Then,
\begin{align}
E_j 1_{\lambda} \otimes 1_{\mu} &= \sum_{A \in A_{\lambda}^{\mu}} \sum_{i=1}^{n}
E_{ij} 1_A :
\bigoplus_{A \in A_{\lambda}^{\mu}} M^A \rightarrow \bigoplus_{A \in A_{\lambda + \alpha_j}^{\mu}} M^A
\label{eq:1}\displaybreak[1]\\
F_j 1_{\lambda} \otimes 1_{\mu} &= \sum_{A \in A_{\lambda}^{\mu}} \sum_{i=1}^{n}
F_{ij} 1_A :
\bigoplus_{A \in A_{\lambda}^{\mu}} M^A \rightarrow \bigoplus_{A \in A_{\lambda - \alpha_j}^{\mu}} M^A
\label{eq:2}\displaybreak[1]\\
1_{\lambda} \otimes E_j 1_{\mu} &= \sum_{A \in A^{\lambda}_{\mu}} \sum_{i=1}^{m}
E_{ij} 1_A : \bigoplus_{A \in A^{\lambda}_{\mu}} M^A \rightarrow \bigoplus_{A \in A^{\lambda}_{\mu + \alpha_j}} M^A
\label{eq:3}\displaybreak[1]\\ \label{eq:4}
1_{\lambda} \otimes F_j 1_{\mu} &= \sum_{A \in A^{\lambda}_{\mu}} \sum_{i=1}^{m}
F_{ij} 1_A : \bigoplus_{A \in A^{\lambda}_{\mu}} M^A \rightarrow \bigoplus_{A \in A^{\lambda}_{\mu - \alpha_j}} M^A
\end{align}
\end{itemize}
\end{thm}

Before we prove Theorem \ref{monprm}, let us calculate some examples.

\subsection{Examples}
In all our examples $d=4$.
It is easy to see that 
\begin{align*}
M^{(3,1)} \otimes M^{(2,2)} &= M^{(2,1,1)} \oplus M^{(2,1,1)}
\displaybreak[1]\\ 
M^{(3,1)} \otimes M^{(3,1)} &= M^{(3,1)} \oplus M^{(2,1,1)} 
\displaybreak[1]\\ 
M^{(3,1)} \otimes M^{(4)} &= M^{(3,1)} 
\end{align*}
Let us show that, in matrix notation,
\begin{equation}
\begin{tikzpicture}[baseline=15, scale=0.75]
\laddercoordinates{1}{1}
\node[below] at (l00) {$3$};
\node[below] at (l10) {$1$};
\ladderI{0}{0};
\ladderI{1}{0};
\end{tikzpicture} 
\otimes
\fork{2}{2}{4}
=
\begin{bmatrix}
\fork{2}{1}{3} 
\begin{tikzpicture}[baseline=15, scale=0.75]
\laddercoordinates{1}{1}
\node[below] at (l00) {$1$};
\ladderI{0}{0};
\end{tikzpicture} \\
\fork{2}{1}{3} 
\begin{tikzpicture}[baseline=15, scale=0.75]
\laddercoordinates{1}{1}
\node[below] at (l00) {$1$};
\ladderI{0}{0};
\end{tikzpicture}
\end{bmatrix}
:
M^{(3,1)} \rightarrow M^{(2,1,1)} \oplus M^{(2,1,1)} 
\end{equation}

Now,
\begin{align*}
\begin{tikzpicture}[baseline=15, scale=0.75]
\laddercoordinates{1}{1}
\node[below] at (l00) {$3$};
\node[below] at (l10) {$1$};
\ladderI{0}{0};
\ladderI{1}{0};
\end{tikzpicture} 
\otimes
\fork{2}{2}{4}
&=
\frac{1}{2}
\left(
\begin{tikzpicture}[baseline=15, scale=0.75]
\laddercoordinates{1}{2}
\node[below] at (l00) {$3$};
\node[below] at (l10) {$1$};
\ladderI{0}{0};
\ladderI{1}{0};
\ladderI{0}{1};
\ladderI{1}{1};
\end{tikzpicture} 
\otimes
\tikz[baseline=15,yscale=0.75]{
\laddercoordinates{1}{2}
\ladderEn{0}{0}{$3$}{$1$}{$1$}
\ladderEn{0}{1}{$2$}{$2$}{$1$}
\node[below] at (l00) {$4$};
\node[below] at (l10) {$0$};
}
\right)
\displaybreak[1]\\
&=
\frac{1}{2}
\left(
\begin{tikzpicture}[baseline=15, scale=0.75]
\laddercoordinates{1}{1}
\node[below] at (l00) {$3$};
\node[below] at (l10) {$1$};
\ladderI{0}{0};
\ladderI{1}{0};
\end{tikzpicture} 
\otimes
\tikz[baseline=15,yscale=0.75]{
\laddercoordinates{1}{1}
\ladderEn{0}{0}{$2$}{$2$}{$1$}
\node[below] at (l00) {$3$};
\node[below] at (l10) {$1$};
}
\right)
\left(
\begin{tikzpicture}[baseline=15, scale=0.75]
\laddercoordinates{1}{1}
\node[below] at (l00) {$3$};
\node[below] at (l10) {$1$};
\ladderI{0}{0};
\ladderI{1}{0};
\end{tikzpicture} 
\otimes
\tikz[baseline=15,yscale=0.75]{
\laddercoordinates{1}{1}
\ladderEn{0}{0}{$3$}{$1$}{$1$}
\node[below] at (l00) {$4$};
\node[below] at (l10) {$0$};
}
\right)
\end{align*}
where the first equality follows from (\ref{band}).
By applying (\ref{eq:3}) and putting the resulting sum in matrix form we get that
\begin{align*}
\begin{tikzpicture}[baseline=15, scale=0.75]
\laddercoordinates{1}{1}
\node[below] at (l00) {$3$};
\node[below] at (l10) {$1$};
\ladderI{0}{0};
\ladderI{1}{0};
\end{tikzpicture} 
\otimes
\tikz[baseline=15,yscale=0.75]{
\laddercoordinates{1}{1}
\ladderEn{0}{0}{$3$}{$1$}{$1$}
\node[below] at (l00) {$4$};
\node[below] at (l10) {$0$};
}
&=
\begin{bmatrix}
\begin{tikzpicture}[baseline=20, scale=0.75]
\laddercoordinates{3}{1}
\node[below] at (l00) {$3$};
\node[below] at (l10) {$0$};
\node[below] at (l20) {$1$};
\node[below] at (l30) {$0$};
\node[above] at (l21) {$0$};
\node[above] at (l31) {$1$};
\ladderI{0}{0};
\ladderEn{2}{0}{}{}{$1$}
\ladderI{1}{0};
\end{tikzpicture} 
\\
\begin{tikzpicture}[baseline=20, scale=0.75]
\laddercoordinates{3}{1}
\node[below] at (l00) {$3$};
\node[below] at (l10) {$0$};
\node[below] at (l20) {$1$};
\node[below] at (l30) {$0$};
\node[above] at (l01) {$2$};
\node[above] at (l11) {$1$};
\ladderI{2}{0};
\ladderEn{0}{0}{}{}{$1$}
\ladderI{3}{0};
\end{tikzpicture} 
\end{bmatrix}
:
M^{
\begin{bmatrix}
    3 & 0  \\
    1 & 0 \\
  \end{bmatrix}
}
\rightarrow
M^{
\begin{bmatrix}
    3 & 0  \\
    0 & 1 \\
  \end{bmatrix}
}
\oplus
M^{
\begin{bmatrix}
    2 & 1  \\
    1 & 0 \\
  \end{bmatrix}
}
\\
&=
\begin{bmatrix}
\begin{tikzpicture}[baseline=20, scale=0.75]
\laddercoordinates{3}{1}
\node[below] at (l00) {$3$};
\node[below] at (l10) {$1$};
\ladderI{0}{0};
\ladderI{1}{0};
\end{tikzpicture} 
\\
\fork{2}{1}{3}
\begin{tikzpicture}[baseline=15, scale=0.75]
\laddercoordinates{1}{1}
\node[below] at (l00) {$1$};
\ladderI{0}{0};
\end{tikzpicture} 
\end{bmatrix}
:
M^{(3,1)} \rightarrow M^{(3,1)} \oplus M^{(2,1,1)} 
\end{align*}

Furthermore,
\begin{align*}
\begin{tikzpicture}[baseline=15, scale=0.75]
\laddercoordinates{1}{1}
\node[below] at (l00) {$3$};
\node[below] at (l10) {$1$};
\ladderI{0}{0};
\ladderI{1}{0};
\end{tikzpicture} 
\otimes
\tikz[baseline=15,yscale=0.75]{
\laddercoordinates{1}{1}
\ladderEn{0}{0}{$2$}{$2$}{$1$}
\node[below] at (l00) {$3$};
\node[below] at (l10) {$1$};
}
&=
\begin{bmatrix}
\begin{tikzpicture}[baseline=20, scale=0.75]
\laddercoordinates{3}{1}
\node[below] at (l00) {$3$};
\node[below] at (l10) {$0$};
\node[below] at (l20) {$0$};
\node[below] at (l30) {$1$};
\node[above] at (l01) {$2$};
\node[above] at (l11) {$1$};
\ladderI{2}{0};
\ladderEn{0}{0}{}{}{$1$}
\ladderI{3}{0};
\end{tikzpicture} 
&
~~
&
\begin{tikzpicture}[baseline=20, scale=0.75]
\laddercoordinates{3}{1}
\node[below] at (l00) {$2$};
\node[below] at (l10) {$1$};
\node[below] at (l20) {$1$};
\node[below] at (l30) {$0$};
\node[above] at (l21) {$0$};
\node[above] at (l31) {$1$};
\ladderI{0}{0};
\ladderEn{2}{0}{}{}{$1$}
\ladderI{1}{0};
\end{tikzpicture} 
\\
\text{{\huge 0}}
&
~~
&
\begin{tikzpicture}[baseline=15, scale=0.75]
\laddercoordinates{3}{1}
\node[below] at (l00) {$2$};
\node[below] at (l10) {$1$};
\node[below] at (l20) {$1$};
\node[below] at (l30) {$0$};
\node[above] at (l01) {$1$};
\node[above] at (l11) {$2$};
\ladderI{2}{0};
\ladderEn{0}{0}{}{}{$1$}
\ladderI{3}{0};
\end{tikzpicture} 
\end{bmatrix}
\\
&=
\begin{bmatrix}
\fork{2}{1}{3}
\begin{tikzpicture}[baseline=15, scale=0.75]
\laddercoordinates{1}{1}
\node[below] at (l00) {$1$};
\ladderI{0}{0};
\end{tikzpicture} 
&
~~
&
\begin{tikzpicture}[baseline=15, scale=0.75]
\laddercoordinates{2}{1}
\node[below] at (l00) {$2$};
\node[below] at (l10) {$1$};
\node[below] at (l20) {$1$};
\ladderI{0}{0};
\ladderI{1}{0};
\ladderI{2}{0};
\end{tikzpicture} 
\\
\text{{\huge 0}}
&
~~
&
\begin{tikzpicture}[baseline=15, scale=0.75]
\laddercoordinates{2}{1}
\node[below] at (l00) {$2$};
\node[below] at (l10) {$1$};
\node[below] at (l20) {$1$};
\node[above] at (l01) {$1$};
\node[above] at (l11) {$2$};
\ladderI{2}{0};
\ladderEn{0}{0}{}{}{$1$}
\end{tikzpicture} 
\end{bmatrix}
\end{align*}
The result follows from the following matrix equation:
\[
\begin{bmatrix}
\fork{2}{1}{3}
\begin{tikzpicture}[baseline=15, scale=0.75]
\laddercoordinates{1}{1}
\node[below] at (l00) {$1$};
\ladderI{0}{0};
\end{tikzpicture} 
&
~~
&
\begin{tikzpicture}[baseline=15, scale=0.75]
\laddercoordinates{2}{1}
\node[below] at (l00) {$2$};
\node[below] at (l10) {$1$};
\node[below] at (l20) {$1$};
\ladderI{0}{0};
\ladderI{1}{0};
\ladderI{2}{0};
\end{tikzpicture} 
\\
\text{{\huge 0}}
&
~~
&
\begin{tikzpicture}[baseline=15, scale=0.75]
\laddercoordinates{2}{1}
\node[below] at (l00) {$2$};
\node[below] at (l10) {$1$};
\node[below] at (l20) {$1$};
\node[above] at (l01) {$1$};
\node[above] at (l11) {$2$};
\ladderI{2}{0};
\ladderEn{0}{0}{}{}{$1$}
\end{tikzpicture} 
\end{bmatrix}
\begin{bmatrix}
\begin{tikzpicture}[baseline=20, scale=0.75]
\laddercoordinates{3}{1}
\node[below] at (l00) {$3$};
\node[below] at (l10) {$1$};
\ladderI{0}{0};
\ladderI{1}{0};
\end{tikzpicture} 
\\
\fork{2}{1}{3}
\begin{tikzpicture}[baseline=15, scale=0.75]
\laddercoordinates{1}{1}
\node[below] at (l00) {$1$};
\ladderI{0}{0};
\end{tikzpicture} 
\end{bmatrix}
=
2
\begin{bmatrix}
\fork{2}{1}{3} 
\begin{tikzpicture}[baseline=15, scale=0.75]
\laddercoordinates{1}{1}
\node[below] at (l00) {$1$};
\ladderI{0}{0};
\end{tikzpicture} \\
\fork{2}{1}{3} 
\begin{tikzpicture}[baseline=15, scale=0.75]
\laddercoordinates{1}{1}
\node[below] at (l00) {$1$};
\ladderI{0}{0};
\end{tikzpicture}
\end{bmatrix}
\]

\subsection{Proof of Theorem \ref{monprm}}

Let $A$ be an $m \times n$ matrix with entries in $\N$, and $\sum_{i,j} A_{ij} = d$. The permutation module
$M^{A}$ has a basis of dissections, $T$, of $\underline{d}$ into $mn$ subsets 
\[
T = \{ 
T^{11}, T^{12}, \ldots, T^{21}, T^{22}, \ldots, T^{mn}
\}
\]
in which $A_{ij} = |T^{ij}|$. Call such a dissection an \textbf{$A$-tabloid} and the set $T^{ij}$ the $(i,j)$-cell of $T$.

By (\cite{eagle}, Lemma 3.1) we have the following decomposition of $S_d$-modules,
\begin{equation}\label{tensortabloid}
M^{\lambda} \otimes M^{\mu} = \bigoplus_{A \in A^{\lambda}_{\mu}} M^{A}
\end{equation}
in which a tensor product $T \otimes T' \in M^{\lambda} \otimes M^{\mu}$ is identified with the $A$-tabloid $T \otimes T' \in  \bigoplus_{A \in A^{\lambda}_{\mu}} M^{A}$ defined $(T \otimes T')^{ij} = T^i \cap T'^j$. 
Equivalently we may index this decomposition
 \begin{equation*}
M^{\lambda} \otimes M^{\mu} = \bigoplus_{A \in A_{\lambda}^{\mu}} M^{A}
\end{equation*}
and identify a tensor product $T \otimes T' \in M^{\lambda} \otimes M^{\mu}$ with the $A$-tabloid $T \otimes T' \in  \bigoplus_{A \in A^{\mu}_{\lambda}} M^{A}$ defined $(T \otimes T')^{ij} = T^j \cap T'^i$.

We just prove equation (\ref{eq:3}). The proofs of the other equations are similar. 

Given any $A$-tabloid $T$,
\[
E_{ij} 1_A (T) = \sum_{k \in T^{i,j+1}} c_{ij, \{k\}} T
\]
where $c_{ij, \{k\}} T$ is the tabloid obtained from $T$ by moving $k$ to the $(i,j)$-cell. 

Fix a $\lambda$-tabloid $T$ and $\mu$-tabloid $T'$. Define the $A$-tabloid $T \otimes T'$ by
$(T \otimes T')^{ij} = T^i \cap T'^j$ as in (\ref{tensortabloid}). 
Note that
\begin{equation}\label{eq:oldman}
c_{i,\{k\}} T \otimes c_{j,\{k\}} T' = c_{ij,\{k\}} (T \otimes T')
\end{equation}
In particular if $k \in T^{i}$ then $T \otimes (c_{j, \{k\}} T') = c_{ij, \{k\}} (T \otimes T')$. Now,
\begin{align*}
(1_{\lambda} \otimes E_j 1_{\mu} ) (T \otimes T') &= \sum_{k \in T'^{j+1}} T \otimes c_{j, \{ k\}} T' \displaybreak[1]\\
&= \sum_{i=1}^{m} \sum_{k \in T^i \cap T'^{j+1}} T \otimes c_{j, \{ k\}} T' \displaybreak[1]\\
&= \sum_{i=1}^{m} \sum_{k \in (T \otimes T')^{i,j+1}} c_{ij,\{k\}} (T \otimes T') &\text{by (\ref{eq:oldman})} \displaybreak[1]\\
&=  \sum_{i=1}^{m}
E_{ij} 1_A (T \otimes T')
\end{align*}
as required.

\section{Brauer Algebras and the CKM Principle}\label{BCKM}

In this section we derive pre-Karoubi presentations of the module categories for the Brauer algebra $\mathcal{B}_{d}^{(- 2n)}$, where $2n \geq d-1$, and the walled Brauer algebra $ \mathcal{B}_{r,s}^{(n)}$, where $n \geq r+s$. We first recall the definition of the Brauer algebra. 

The Brauer algebra $\mathcal{B}_{d}^{(\delta)}$, $\delta \in \Z$, has a $\C$-basis of $d$-diagrams. A $d$-diagram consists of $2d$ dots arranged with $d$ dots in each of two rows, and a pairing of these dots. For example, the following is a $6$-diagram,
 
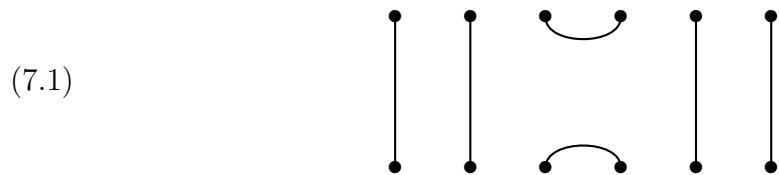
\begin{equation}\label{bdiagram}
 \tikz[baseline=2.5em]{
    \foreach \x in {1,...,6}
    {
      \draw (\x,0) node {$\bullet$};
      \draw (\x,2) node {$\bullet$};
    }
    \draw[thick] (1,2) to[out=-90,in=90] (1,0);
    \draw[thick] (2,2) to[out=-90,in=90] (2,0);
    \draw[thick] (3,2) to[out=-90,in=-90] (4,2);
    \draw[thick] (3,0) to[out=90,in=90] (4,0);
    \draw[thick] (5,2) to[out=-90,in=90] (5,0);
    \draw[thick] (6,2) to[out=-90,in=90] (6,0);
  }
 \end{equation}
Given two $d$-diagrams $x$, $y$, the product $xy$ in $\mathcal{B}_{d}^{(\delta)}$ is defined as $\delta^{l} z$, where $z$ is the diagram obtained by vertical juxtaposition of $x$ above $y$ (identifying the bottom row of dots in $x$ with the top row of dots in $y$) and removing any resulting closed loops; and $l$ is the number of such closed loops removed.

Index the vertices on each row from left to right. Define the element $c_{ij} \in \mathcal{B}_{d}^{(\delta)}$ to be the $d$-diagram in which the $i$-th vertex on the top (respectively bottom) row is paired with the $j$-th vertex on the top (respectively bottom) row, and all other vertices are paired with the vertex directly above/below itself. For example,  (\ref{bdiagram}) is the 6-diagram $c_{34}$.

Identify $S_d \subset \mathcal{B}_{d}^{(\delta)}$ with the $d$-diagrams in which all vertices are paired with a vertex in a different row. More precisely, the permutation $\sigma \in S_d$ is identified with the $d$-diagram $\sigma \in \mathcal{B}_{d}^{(\delta)}$ defined: the $i$-th vertex on the bottom row of $\sigma \in \mathcal{B}_{d}^{(\delta)}$ is paired with the $j$-th vertex on the top row if and only if $\sigma \in S_d$ maps $i$ to $j$.

\subsection{Schur-Weyl Duality in types C and D}\label{sym}

We define the Lie algebra $\mathfrak{sp}_{2n}$ and $\mathfrak{o}_{2n}$ with respect to the skew-symmeric and symmetric forms defined respectively by the matrices 
$
\begin{pmatrix}
    0 & I_n  \\
     - I_n & 0   
\end{pmatrix}
$
and
$
\begin{pmatrix}
    0 & I_n  \\
     I_n & 0   
\end{pmatrix}
$
,
 where $I_n$ is the $n \times n$ identity matrix. We then get,
 \begin{align*}
 \sp_{2n} &= \{ 
\begin{pmatrix}
    A & B  \\
     C & -A^{t}   
\end{pmatrix}
\in
\gl_{2n}
|
B=B^{t} , C=C^{t}
 \}
 \\
 \mathfrak{o}_{2n}
 &= \{ 
\begin{pmatrix}
    A & B  \\
     C & -A^{t}   
\end{pmatrix}
\in
\gl_{2n}
|
B=-B^{t} , C=-C^{t}
 \}
 \end{align*}
 
Let us take $\g$ to be either $\sp_{2n}$ or $\mathfrak{o}_{2n}$.
Both $\sp_{2n}$ and $\mathfrak{o}_{2n}$ have the Cartan subalgebra 
$\mathfrak{h}$ spanned by the elements $Z_{i} = e_{ii} - e_{n+i,n+i}$. Write $\varepsilon_{i} \in \mathfrak{h}^{*}$ for the element dual to $Z_i$. We write weights of $\g$-modules by their vector coordinates with respect to the basis $\{ \varepsilon_{i} \}_{i=1, \ldots, n}$ of $\mathfrak{h}^{*}$. The set of integral weights of $\g$ is $\Z^n$ if $\g=\sp_{2n}$ and $\Z^n \cup ((\frac{1}{2}, \ldots, \frac{1}{2})+\Z^n)$ if $\g=\mathfrak{o}_{2n}$.
The $\g$-dominant weights are those weights 
$(\lambda_1, \ldots , \lambda_n)$ in which 
\begin{align*}
\lambda_1 \geq \cdots \geq \lambda_n &\geq 0 &\text{if $\g = \sp_{2n}$}
\displaybreak[1]\\
\lambda_1 \geq \cdots \geq \lambda_{n-1} &\geq  |\lambda_n| &\text{if $\g = \mathfrak{o}_{2n}$}
\end{align*}

Let $\{v_1 , \ldots , v_n , v_{-1}, \ldots, v_{-n} \}$ denote the standard basis of $\C^{2n}$. Then with respect to the skew-symmetric form defined above, $v_{-i}$ is dual to $v_{i}$. Furthermore, under the natural action of $\sp_{2n}$ and $\mathfrak{o}_{2n}$ on $\C^{2n}$, the vector $v_{\pm i}$ spans the $ \pm \varepsilon_i$ weight space of $\C^{2n}$. In either case $\g = \sp_{2n}$ or $\g = \mathfrak{o}_{2n}$, the $\g$-module $W= \bigotimes^{d} \C^{2n}$ has the set of weights
\[
\Pi (W) = \bigcup_{j=0}^{\floor{\frac{d}{2}}}  
\{
\lambda \in \Z^n | \sum_{i} |\lambda_i| = d-2j
\}
\]

There is a right action of $\mathcal{B}_{d}^{(-2n)}$ on $W$ commuting with the natural action of $\mathfrak{sp}_{2n}$. For $i,j \in \{\pm 1 , \ldots , \pm n \}$ define
\[
\epsilon_{ij} =
\begin{cases}
1 &\text{if $i=-j$ and $i >0$,} \\
-1 &\text{if $i=-j$ and $i<0$,} \\
0 &\text{otherwise}
\end{cases}
\]
and let $\varepsilon : S_d \rightarrow \{ \pm1\}$ be the sign map. The right action of $\mathcal{B}_{d}^{(-2n)}$ on $W$ is given by,
\begin{align*}
(v_{i_1} \otimes \cdots \otimes v_{i_d}) \cdot \sigma &= \varepsilon (\sigma) v_{i_{\sigma(1)}} \otimes \cdots \otimes v_{i_{\sigma(d)}} & \text{for $\sigma \in S_d \subset \mathcal{B}_{d}^{(-2n)}$} \\
(v_{i_1} \otimes \cdots \otimes v_{i_d}) \cdot c_{j,j+1} &=
\epsilon_{i_j, i_{j+1}} v_{i_1} \otimes \cdots \otimes v_{i_{j-1}} \otimes \\&~~~~
\left(
\sum_{k=1}^{n} (v_{-k} \otimes v_{k} - v_{k} \otimes v_{-k})
\right)
\otimes v_{i_{j+2}}
\otimes \cdots \otimes v_{i_n}
\end{align*}

The $(\sp_{2n}, \mathcal{B}_{d}^{(-2n)})$-bimodule $W=\bigotimes^d \C^{2n}$ has a saturated multiplicity-free decomposition (\cite{doty3}, Proposition 1.1.3),
\[
W = \bigoplus_{j=0}^{\floor{\frac{d}{2}}} \bigoplus_{\lambda 
\in \Lambda^{+} (n , d - 2j)} V^{\lambda} \otimes \Delta^{\lambda}
\]
where $V^{\lambda}$ is the irreducible $\sp_{2n}$-module of highest weight $\lambda$, and the $\Delta^{\lambda}$ 
are irreducible $\mathcal{B}_{d}^{(-2n)}$ modules. If $2n \geq d-1$ then $\mathcal{B}_{d}^{(-2n)} = \End_{\sp_{2n}} W$ \cite{brown} and so $\{\Delta^{\lambda}\}_{\lambda \in \Lambda^{+} (n , d - 2j)}$ is the complete set of irreducible $\mathcal{B}_{d}^{(-2n)}$-modules up to isomorphism. Thus, if $2n \geq d-1$, then a pre-Karoubi presentation of $\Rep \mathcal{B}_{d}^{(-2n)}$ can be obtained from the CKM Principle. We do this in Proposition \ref{symsw}.

There is a right action of $\mathcal{B}_{d}^{(2n)}$ on $W$ commuting with the natural action of $\mathfrak{o}_{2n}$. This action is defined,
\begin{align*}
(v_{i_1} \otimes \cdots \otimes v_{i_d}) \cdot \sigma &= v_{i_{\sigma(1)}} \otimes \cdots \otimes v_{i_{\sigma(d)}} & \text{for $\sigma \in S_d \subset \mathcal{B}_{d}^{(2n)}$} \\
(v_{i_1} \otimes \cdots \otimes v_{i_d}) \cdot c_{j,j+1} &=
\delta_{i_j, - i_{j+1}} v_{i_1} \otimes \cdots \otimes v_{i_{j-1}} \otimes \\&~~~~
\left(
\sum_{k=1}^{n} (v_{k} \otimes v_{-k} + v_{-k} \otimes v_{k})
\right)
\otimes v_{i_{j+2}}
\otimes \cdots \otimes v_{i_n}
\end{align*}

(\cite{doty3}, Proposition 1.1.3) shows that the $(\mathfrak{o}_{2n}, \mathcal{B}_{d}^{(2n)})$-bimodule $W$ has a saturated multiplicity-free decomposition. Furthermore \cite{brown} shows that if $2n \geq d-1$ then $\mathcal{B}_{d}^{(2n)} = \End_{\mathfrak{o}_{2n}} W$. By the CKM Principle one can directly obtain a pre-Karoubi presentation of $\Rep \mathcal{B}_{d}^{(2n)}$ when $2n \geq d-1$. We do not do this here.

\begin{rem}
For all $\lambda \in \Pi (W)$, a basis for the $\lambda$-weight space, $W_{\lambda}$, of $W$ (as either an $\sp_{2n}$ or $\mathfrak{o_{2n}}$ module) is given by the $S_d$-orbit of the vectors of the form
$
v_{i_1} \otimes \cdots \otimes v_{i_r} \otimes v_{-{j_1}} \otimes \cdots \otimes v_{-j_s}
$,
where each $i_k > 0$, $j_k >0$.
In particular $W_{\lambda}$ has the following direct sum decomposition of $S_d$-modules,
\begin{align*}
W_{\lambda} &\cong \bigoplus_{(r,s)\vDash d} \bigoplus_{\substack{\mu \in \Lambda (n , r)\\ \nu \in \Lambda (n,s) \\ \lambda = \mu - \nu}}
\left( M^{\mu} \boxtimes M^{\nu} \right) \otimes_{\C [S_r \times S_s]} \C [S_d]
\\
\nonumber
&= \bigoplus_{(r,s)\vDash d} \bigoplus_{\substack{\mu \in \Lambda (n , r)\\ \nu \in \Lambda (n,s) \\ \lambda = \mu - \nu}}
M^{(\mu, \nu)}
\end{align*}
This agrees with the weight space decomposition of $\bigotimes^d \C^{2n}$ given in \cite{doty2}.
\end{rem}

\subsection{A pre-Karoubi presentation of $\Rep \mathcal{B}_{d}^{(-2n)}$}
Our next goal is to derive a diagrammatic pre-Karoubi presentation for $\Rep \mathcal{B}_{d}^{(-2n)}$ where $2n \geq d-1$. For this we need a diagrammatic presentation of $\dU \sp_{2n}$.

Define elements of $\sp_{2n}$,
\begin{align*}
E_i = e_{i,i+1} - e_{n+i+1,n+i} &\qquad \text{ and } \qquad F_i = e_{i+1,i} - e_{n+i,n+i+1} 
\displaybreak[1]\\
X_{j} =e_{j,n+j} &\qquad \text{ and } \qquad Y_{j} =e_{n+j,j} 
\displaybreak[1]\\
Z_{j} = e_{jj} - e_{n+j,n+j} &\qquad \text{ and } \qquad H_i =Z_{i}- Z_{i+1}
\end{align*}
where $ 1 \leq i \leq n-1$ and $1 \leq j \leq n$.
The standard Chevalley generators of $\sp_{2n}$ are 
\[
\{E_i, F_i, H_i \}_{i=1,\ldots, n-1} \cup \{ X_n, Y_{n}, Z_n\}
\]
corresponding to the standard base $\{\varepsilon_1 - \varepsilon_2 , \ldots, \varepsilon_{n-1} - \varepsilon_{n}, 2 \varepsilon_n \}$. The elements $X_i$, $Y_i$, for $i < n$ can be defined in terms of these generators as
\begin{align}\label{XY}
X_{i} = \frac{1}{2} [E_i , [E_i , X_{i+1}]] 
\qquad
\text{ and }
\qquad
Y_{i} = \frac{1}{2} [F_i , [F_i , Y_{i+1}]]
\end{align}

The category $\dU \sp_{2n}$ has presentation,
\begin{itemize}
\item
Objects: Sequences in $\Z^n$. 
\item
Morphisms: Write $\varepsilon_i$ for the sequence in $\Z^n$ with a 1 in the $i$-th position and 0's elsewhere. The morphisms are generated by $E^{(r)}_{i} 1_{\lambda} \in 1_{\lambda + r \alpha_i} \dot{\mathcal{U}} \sp_{2n}  1_{\lambda}$ and $F^{(r)}_{i} 1_{\lambda} \in 1_{\lambda - r \alpha_i} \dot{\mathcal{U}} \mathfrak{sp}_{2n}  1_{\lambda}$ for $ 1 \leq i \leq n-1$, and $X_{j} 1_{\lambda} \in 1_{\lambda + 2 \varepsilon_j} \dot{\mathcal{U}} \sp_{2n}  1_{\lambda}$ and $Y_{j} 1_{\lambda} \in 1_{\lambda - 2 \varepsilon_j} \dot{\mathcal{U}} \sp_{2n}  1_{\lambda}$ for $1 \leq j \leq n$. These satisfy relations (\ref{rel:1}), (\ref{rel:2}), (\ref{rel:3}), (\ref{rel:4}), (\ref{rel:5}) as well as,
\begin{align}
X_j Y_{j} 1_{\lambda} &= Y_{j} X_{j} 1_{\lambda} + \lambda_j 1_{\lambda}
\label{Usp:1}
\displaybreak[1]\\
(X_i + E_i X_{i+1} E_{i})1_{\lambda} &= (E^{(2)}_i X_{i+1} + X_{i+1} E^{(2)}_{i} ) 1_{\lambda}
\label{Usp:2}
\displaybreak[1]\\
 2 X_{i+1} E_{i} X_{i+1} 1_{\lambda} &= (X_{i+1}^{2} E_{i} + E_{i} X_{i+1}^{2}) 1_{\lambda}
 \label{Usp:3}
\displaybreak[1]\\
E_{i} X_{i} 1_{\lambda} &= X_{i} E_{i} 1_{\lambda}
\label{Usp:4}
\displaybreak[1]\\
E_i Y_{j} 1_{\lambda} &= Y_{j} E_{i} 1_{\lambda} &\text{if $i \neq j$}
\label{Usp:5}
\displaybreak[1]\\
E_i X_{j} 1_{\lambda} &= X_{j} E_{i} 1_{\lambda} &\text{if $|i - j| >1$}
\label{Usp:6}
\displaybreak[1]\\
X_i Y_j 1_{\lambda} &= Y_{j} X_i 1_{\lambda} &\text{if $i \neq j$}
\label{Usp:7}
\end{align}
together with the equations formed by interchanging each $E_i$ with $F_i$ and each $X_j$ with $Y_{j}$.
\end{itemize}

Indeed equation (\ref{Usp:2}) corresponds to the definitions of the redundant generators in (\ref{XY}).
The standard Cartan matrix for $\sp_{2n}$ is
\[
\begin{pmatrix}
    2 & -1 & 0 & \cdots  & 0 \\
     -1 & 2 & -1 & \ddots  & \vdots  \\
     0 & \ddots & \ddots & -1 & 0\\
    \vdots & \ddots & -1 & 2  & -1 \\
    0 & \cdots & 0 & -2  & 2
\end{pmatrix}
\]
The Serre relation $\ad^2 (X_n) E_{n-1} 1_{\lambda}= 0$ is a special case of (\ref{Usp:3}). The Serre relation 
$\ad^{3} (E_{n-1}) X_n 1_{\lambda} =0$ follows from (\ref{Usp:4}) and (\ref{Usp:2}). It is not hard to check that each 
relation (\ref{Usp:1})-(\ref{Usp:7}) must hold in $\dU \sp_{2n}$ and each Serre relation, (\ref{Ug1}) - (\ref{Ug3}), defining $\dU \sp_{2n}$  follows from the relations  (\ref{rel:1}) -  (\ref{rel:5}), (\ref{Usp:1})-(\ref{Usp:7}). Hence this is indeed a presentation of $\dU \sp_{2n}$.

We can identify morphisms in $\dU \sp_{2n}$ as web diagrams in the following way. Define an \textbf{$n$-ladder with bells} to be a diagram formed by vertical stacking of $n$-ladders and the following diagrams consisting of $n$ vertical strands labelled by integers,
\begin{align}
\begin{tikzpicture}[baseline=20, xscale=0.75]
\laddercoordinates{3}{1}
\node[below] at (l00) {$\lambda_{i{-}1}$};
\node[below] at (l10) {$\lambda_i$};
\node[below] at (l20) {$\lambda_{i{+}1}$};
\node[above] at (l11) {$\lambda_i{+}2$};
\node at ($(l00)+(-1,1)$) {$\cdots$};
\ladderI{0}{0};
\ladderI{1}{0};
\ladderI{2}{0};
\node at ($(l20)+(1,1)$) {$\cdots$};
\node at ($(l10)+(0,0.8)$) [amp, draw, scale=0.75] {};
\end{tikzpicture} 
\qquad
\text{and}
\qquad
\begin{tikzpicture}[baseline=20, xscale=0.75]
\laddercoordinates{3}{1}
\node[below] at (l00) {$\lambda_{i{-}1}$};
\node[below] at (l10) {$\lambda_i$};
\node[below] at (l20) {$\lambda_{i{+}1}$};
\node[above] at (l11) {$\lambda_i{-}2$};
\node at ($(l00)+(-1,1)$) {$\cdots$};
\ladderI{0}{0};
\ladderI{1}{0};
\ladderI{2}{0};
\node at ($(l20)+(1,1)$) {$\cdots$};
\node at ($(l10)+(0,0.8)$) [coamp, draw, scale=0.75] {};
\end{tikzpicture} 
\end{align}
We call the above bivalent vertices, \textbf{bells} and \textbf{cobells} respectively.
We identify the morphisms $E^{(r)}_{i} 1_{\lambda}$ and $F^{(r)}_{i} 1_{\lambda}$ in $\dU \sp_{2n}$ with $n$-ladders as in (\ref{glnladder}), and identify
\begin{align}\label{sp2nladder}
X_i \one_{\lambda} &=
\begin{tikzpicture}[baseline=20, xscale=0.75]
\laddercoordinates{3}{1}
\node[below] at (l00) {$\lambda_{i{-}1}$};
\node[below] at (l10) {$\lambda_i$};
\node[below] at (l20) {$\lambda_{i{+}1}$};
\node[above] at (l11) {$\lambda_i{+}2$};
\node at ($(l00)+(-1,1)$) {$\cdots$};
\ladderI{0}{0};
\ladderI{1}{0};
\ladderI{2}{0};
\node at ($(l20)+(1,1)$) {$\cdots$};
\node at ($(l10)+(0,0.8)$) [amp, draw, scale=0.75] {};
\end{tikzpicture} 
\qquad
\text{and}
\qquad
Y_i \one_{\lambda} &=
\begin{tikzpicture}[baseline=20, xscale=0.75]
\laddercoordinates{3}{1}
\node[below] at (l00) {$\lambda_{i{-}1}$};
\node[below] at (l10) {$\lambda_i$};
\node[below] at (l20) {$\lambda_{i{+}1}$};
\node[above] at (l11) {$\lambda_i{-}2$};
\node at ($(l00)+(-1,1)$) {$\cdots$};
\ladderI{0}{0};
\ladderI{1}{0};
\ladderI{2}{0};
\node at ($(l20)+(1,1)$) {$\cdots$};
\node at ($(l10)+(0,0.8)$) [coamp, draw, scale=0.75] {};
\end{tikzpicture} 
\end{align}

It is not hard to show that $\dU \sp_{2n}$ has presentation,
\begin{itemize}
\item
Objects: Sequences in $\Z^n$.
\item
Morphisms: Morphisms from $\lambda$ to $\mu$ are $\C$-linear combinations of $n$-ladders with bells satisfying relations (\ref{eq:IHlad}), (\ref{eq:IHlad2}), (\ref{eq:EE}), (\ref{eq:EF=FE}), (\ref{eq:REP4}), invariance under upwards-orientation preserving planar isotopy, as well as the following relations,
\begin{align}
\begin{tikzpicture}[baseline=40]
\ladderI{0}{0};
\ladderI{0}{1};
\node[below] at (l00) {$k$};
\node at (0,1) [coamp, draw, scale=0.75] {};
\node at (0,1.5) [amp, draw, scale=0.75] {};
\end{tikzpicture}
~
&=
~
\begin{tikzpicture}[baseline=40]
\ladderI{0}{0};
\ladderI{0}{1};
\node[below] at (l00) {$k$};
\node at (0,1) [amp, draw, scale=0.75] {};
\node at (0,1.5) [coamp, draw, scale=0.75] {};
\end{tikzpicture}
~
{+}
~
k
\cdot
\begin{tikzpicture}[baseline=40]
\ladderI{0}{0};
\ladderI{0}{1};
\node[below] at (l00) {$k$};
\end{tikzpicture}
\\
\begin{tikzpicture}[baseline=40]
\ladderI{0}{0};
\ladderI{1}{0};
\ladderI{0}{1};
\ladderI{1}{1};
\node at ($(l01)$) [amp, draw, scale=0.75] {};
\end{tikzpicture} 
+
\begin{tikzpicture}[baseline=40]
\ladderF{0}{0}{}{}
\ladderF{0}{1}{}{}
\node at ($(l11)$) [amp, draw, scale=0.75] {};
\end{tikzpicture} 
&=
\begin{tikzpicture}[baseline=40]
\ladderI{0}{0};
\ladderI{1}{0};
\ladderFn{0}{1}{}{}{2}
\node at ($(l11)$) [amp, draw, scale=0.75] {};
\end{tikzpicture}
+
\begin{tikzpicture}[baseline=40]
\ladderI{0}{1};
\ladderI{1}{1};
\ladderFn{0}{0}{}{}{2}
\node at ($(l11)$) [amp, draw, scale=0.75] {};
\end{tikzpicture}
\displaybreak[1]
\\
2
\begin{tikzpicture}[baseline=40, yscale=0.75]
\laddercoordinates{1}{3}
\ladderI{0}{0};
\ladderI{0}{2};
\ladderI{1}{0};
\ladderI{1}{2};
\ladderF{0}{1}{}{}
\node at ($(l12)$) [amp, draw, scale=0.75] {};
\node at ($(l11)-(0,0.5)$) [amp, draw, scale=0.75] {};
\end{tikzpicture}
&=
\begin{tikzpicture}[baseline=40, yscale=0.75]
\laddercoordinates{1}{3}
\ladderI{0}{0};
\ladderI{0}{2};
\ladderI{1}{0};
\ladderI{1}{2};
\ladderF{0}{1}{}{}
\node at ($(l13)-(0,0.75)$) [amp, draw, scale=0.75] {};
\node at ($(l12)$) [amp, draw, scale=0.75] {};
\end{tikzpicture}
+
\begin{tikzpicture}[baseline=40, yscale=0.75]
\laddercoordinates{1}{3}
\ladderI{0}{0};
\ladderI{0}{2};
\ladderI{1}{0};
\ladderI{1}{2};
\ladderF{0}{1}{}{}
\node at ($(l10)+(0,0.75)$) [amp, draw, scale=0.75] {};
\node at ($(l11)$) [amp, draw, scale=0.75] {};
\end{tikzpicture}
\displaybreak[1]
\\
\begin{tikzpicture}[baseline=40]
\ladderI{0}{0};
\ladderI{1}{0};
\ladderF{0}{1}{}{}
\node at ($(l01)$) [amp, draw, scale=0.75] {};
\end{tikzpicture}
&=
\begin{tikzpicture}[baseline=40]
\ladderI{0}{1};
\ladderI{1}{1};
\ladderF{0}{0}{}{}
\node at ($(l01)$) [amp, draw, scale=0.75] {};
\end{tikzpicture}
\displaybreak[1]
\\
\begin{tikzpicture}[baseline=40]
\ladderI{0}{0};
\ladderI{1}{0};
\ladderE{0}{1}{}{}
\node at ($(l11)$) [amp, draw, scale=0.75] {};
\end{tikzpicture}
&=
\begin{tikzpicture}[baseline=40]
\ladderI{0}{1};
\ladderI{1}{1};
\ladderE{0}{0}{}{}
\node at ($(l11)$) [amp, draw, scale=0.75] {};
\end{tikzpicture}
\end{align}
together with the relations formed by reversing the orientation of horizontal strands and interchanging bells and cobells. We interpret unlabeled horizontal strands as being labeled by 1, and unlabeled vertical strands as being labeled by arbitrary compatible labels. These diagrams are to be interpreted as having some number of vertical strands to the left and right.
\end{itemize}

We call diagrams depicting morphisms in this presentation of $\dU \sp_{2n}$, \textbf{$\sp_{2n}$-ladders}.

\begin{rem}
For $\sp_{2n}$-ladders to be invariant under planar isotopy it remains to check that the relation 
\begin{align*}
X_i X_j 1_{\lambda} = X_{j} X_i 1_{\lambda} &&\text{if $i \neq j$}
\end{align*}
holds in $\dU \sp_{2n}$. Indeed this relation can be derived from the other relations. Alternatively, this must hold in $\dU \sp_{2n}$ since $[X_i, X_j] =0$ in $\sp_{2n}$.
\end{rem}

\begin{rem}
A different diagrammatic presentation of $\dU \sp_{2n}$ can be obtained by choosing a different base for the weight lattice of $\sp_{2n}$ and playing this same game. Diagrammatic presentations of $\dU \mathfrak{o}_{2n}$ can be also obtained in this way -  cf. \cite{bcdweb}. We do not do this here.
\end{rem}

The following proposition follows by the CKM principle.
\begin{prop}\label{symsw}
Let $\dot{\mathcal{S}}_{C} (n,d)$ be the full subcategory of $\Rep \mathcal{B}_{d}^{(-2n)}$ whose objects are the weight spaces of $\bigotimes^d \C^{2n}$. Then $\dot{\mathcal{S}}_{C} (n,d)$ is isomorphic to the category defined
\begin{itemize}
\item
Objects: Sequences $\lambda \in \Z^n$ in which $\sum_i |\lambda_i| = d - 2j$ for some $0 \leq j \leq \floor{\frac{d}{2}}$.
\item
Morphisms: Morphisms from $\lambda$ to $\mu$ are $\C$-linear combinations of $\sp_{2n}$-ladders connecting $\lambda$ to $\mu$, and satisfying the relation
\begin{align*}
\begin{tikzpicture}[baseline=20, scale=0.75]
\laddercoordinates{3}{1}
\node[below] at (l00) {$\nu_{1}$};
\node[below] at (l20) {$\nu_{n}$};
\ladderI{0}{0};
\node at ($(l10)+(0,1)$) {$\cdots$};
\ladderI{2}{0};
\end{tikzpicture} 
&= 0
&\text{if $\sum_i |\nu_i| \neq d - 2j$ for any $0 \leq j \leq \floor{\frac{d}{2}}$}
\end{align*}
\end{itemize}
Moreover, if $2n \geq d-1$ then $\dot{\mathcal{S}}_{C} (n,d)$ is a pre-Karoubi subcategory of 
$\Rep \mathcal{B}_{d}^{(-2n)}$.
\end{prop}

There is a bifunctor $\cdot || \cdot : \dot{\mathcal{S}}_{C} (n,d) \times \dot{\mathcal{S}}_{C} (m,d') \rightarrow \dot{\mathcal{S}}_{C} (n+m,d+d')$ defined by horizontal juxtapostion of $\sp_{2n}$-ladders. 

\begin{prop} 
The bifunctor $\cdot || \cdot $ is equal to the bifunctor 
\[
\Ind_{\mathcal{B}_{d}^{(-2n)} \times \mathcal{B}_{d'}^{(-2m)}}^{\mathcal{B}_{d+d'}^{(-2n-2m)}} (\cdot \boxtimes \cdot) : \dot{\mathcal{S}}_{C} (n,d) \times \dot{\mathcal{S}}_{C} (m,d') \rightarrow \dot{\mathcal{S}}_{C} (n+m,d+d').
\]
\end{prop}

\begin{proof}
This can be proven in the same way as Proposition \ref{bjm}.
\end{proof}

\subsection{Mixed Schur-Weyl Duality and the CKM principle}\label{msdckm}

Consider the $\gl_n$-module
\[
V^{r,s} := \textstyle\bigotimes^{r} \C^n \otimes \textstyle\bigotimes^{s} \C^{*n}
\]
Benkart et. al. \cite{mixsw} described the commutant of $V^{r,s}$, for $n \geq r+s$, to be the walled Brauer algebra $\mathcal{B}_{r,s}^{(n)}$. This is the subalgebra of $\mathcal{B}_{r+s}^{(n)}$ spanned by those Brauer diagrams in which
\begin{enumerate}
\item
All horizontal edges pair one of the $r$ leftmost vertices with one of the $s$ rightmost vertices.
\item
No vertical edges pair one of the $r$ leftmost vertices with one of the $s$ rightmost vertices.
\end{enumerate} 

Write $\{ v_{-1}, \ldots, v_{-n}\}$ for the standard basis of $\C^{*n}$ dual to the standard basis 
$\{ v_1 , \ldots , v_n\}$ of $\C^n$. 
The right action of $\mathcal{B}_{r,s}^{(n)}$ on $V^{r,s}$ is defined in the same way as the right action of $\mathcal{B}_{d}^{(2n)}$ on $\bigotimes^{d} \C^{2n}$. 

Use square brackets for pairs of sequences. Define the set
\[
\Pi (r,s) := \{ [\mu, \nu] ~| ~\mu \vdash r, ~\nu \vdash s, ~ l(\mu) + l(\nu) \leq n\}
\]
\cite{mixsw} showed that $V^{r,s}$ has a multiplicity free decomposition
\[
V^{r,s} = \bigoplus_{j=0}^{\min\{ r, s\}} \bigoplus_{[\mu , \nu] \in \Pi (r-j, s-j)} V^{[\mu, \nu]} \otimes \Delta^{[\mu, \nu]}
\]
where the $\Delta^{[\mu, \nu]}$ are irreducible $\mathcal{B}_{r,s}^{(n)}$ modules and $V^{[\mu, \nu]}$ is the irreducible $\gl_n$-module with highest weight
$
(
\mu_1 , \ldots , \mu_{l(\mu)}, 0 , \ldots , 0 , - \nu_{l(\nu)} , \ldots , - \nu_1 
) \in \Z^n
$.
It is not hard to check that $V^{r,s}$ has weights
\[
\Pi (V^{r,s}) = \bigcup_{j=0}^{\min\{ r, s\}} \{
\lambda \in \Z^n | \sum_{i} |\lambda_i|=r+s-2j , \sum_i \lambda_i = r-s
 \}
\]
and that $V^{r,s}$ is saturated. Furthermore \cite{mixsw} showed that if $n \geq r+s$ then 
$\mathcal{B}_{r,s}^{(n)} = \End_{\gl_n} V^{r,s}$. Hence we have the following proposition.

\begin{prop}\label{msw}
Let $\dot{\mathcal{S}} (n;r,s)$ be the full subcategory of $\Rep \mathcal{B}_{r,s}^{(n)}$ whose objects are the weight spaces of $V^{r,s}$. Then $\dot{\mathcal{S}} (n;r,s)$ is isomorphic to the category defined
\begin{itemize}
\item
Objects: Sequences $\lambda \in \Pi (V^{r,s})$.
\item
Morphisms: Morphisms from $\lambda$ to $\mu$ are $\C$-linear combinations of $\gl_n$-ladders connecting $\lambda$ to $\mu$, and satisfying the relation
\begin{align*}
\begin{tikzpicture}[baseline=20, scale=0.75]
\laddercoordinates{3}{1}
\node[below] at (l00) {$\nu_{1}$};
\node[below] at (l20) {$\nu_{n}$};
\ladderI{0}{0};
\node at ($(l10)+(0,1)$) {$\cdots$};
\ladderI{2}{0};
\end{tikzpicture} 
&= 0
&\text{if $\nu \notin \Pi (V^{r,s})$}
\end{align*}
\end{itemize}
Moreover, if $n \geq r+s$ then $\dot{\mathcal{S}} (n; r,s)$ is a pre-Karoubi subcategory of 
$\Rep \mathcal{B}_{r,s}^{(n)}$.
\end{prop}

By giving a proof similar to that of Proposition \ref{bjm} we get the following result.

\begin{prop}
The bifunctor $\dot{\mathcal{S}} (n;r,s) \times \dot{\mathcal{S}} (m;r',s') \rightarrow \dot{\mathcal{S}} (n+m;r+r',s+s')$ defined by horizontal juxtaposition of $\gl_n$-ladders is equal to the bifunctor
\[
\Ind_{\mathcal{B}_{r,s}^{(n)} \times \mathcal{B}_{r',s'}^{(m)}}^{\mathcal{B}_{r+r', s+s'}^{(m+n)}} (\cdot \boxtimes \cdot) : \dot{\mathcal{S}} (n;r,s) \times \dot{\mathcal{S}} (m;r',s') \rightarrow \dot{\mathcal{S}} (n+m;r+r',s+s').
\]
\end{prop}

\begin{rem}
The $\lambda$-weight space, $V_{\lambda}^{r,s}$, of $V^{r,s}$ has the following direct sum decomposition of $\C [S_{r}] \times \C[S_{s}]$-modules,
\begin{equation}\label{mixweight}
V_{\lambda}^{r,s} = \bigoplus_{\substack{\mu \in \Lambda (n , r)\\ \nu \in \Lambda (n,s) \\ \lambda = \mu - \nu}}
M^{\mu} \boxtimes M^{\nu}
\end{equation}
A more explicit description of these weight spaces as $\mathcal{B}_{r,s}^{(n)}$-modules can be given in the case $r=s$.
Indeed there is an isomorphism of $\gl_n$-bimodules $V^{d,d} \rightarrow \bigotimes^{d} \gl_n$ defined
\[
v_{i_1} \otimes \cdots \otimes v_{i_d} \otimes v_{-{j_1}} \otimes \cdots \otimes v_{-j_d}
\mapsto
e_{i_1, j_1} \otimes \cdots \otimes e_{i_d, j_d}
\]
A basis element $v = e_{i_1, j_1} \otimes \cdots \otimes e_{i_d, j_d}$ of $\bigotimes^{d} \gl_n$ corresponds to a dissection
\begin{align*}
T_v &= \{ T^{11}_v , \ldots , T^{1n}_v , T^{21}_v , \ldots , T^{nn} \},
&\text{where $T^{ij}_v = \{k | e_{i_k, j_k} =e_{ij} \}$} 
\end{align*}
Hence the $\lambda$-weight space, $N^{\lambda}$, of $\bigotimes^{d} \gl_n$ has a basis of $A$-tabloids, where 
\[
A \in \bigcup_{\substack{\mu, \nu \in \Lambda (n , d)\\ \lambda = \mu - \nu}} A^{\mu}_{\nu}.
\]
A cute exercise we leave with the reader is to describe the $\mathcal{B}_{d,d}^{(n)}$ action on these tabloids.
\end{rem}

\bibliographystyle{alpha}
\bibliography{bibliography/bibliography}

\vspace{1cm}

\end{document}